%% file: SpaceSupport.tex
\documentclass[a4paper,12pt]{amsart}
\input{DefinitionsSpaceSupport.tex} 

\begin{document}

\title{Stratifying derived categories of cochains on certain spaces}
\author{Shoham Shamir}
\address{Department of Mathematics, University of Bergen, 5008 Bergen, Norway}
\email{shoham.shamir@math.uib.no}
\date{\today}

\begin{abstract}
In recent years, Benson, Iyengar and Krause have developed a theory of stratification for compactly generated triangulated categories with an action of a graded commutative Noetherian ring. Stratification implies a classification of localizing and thick subcategories in terms of subsets of the prime ideal spectrum of the given ring. In this paper two stratification results are presented: one for the derived category of a commutative ring-spectrum with polynomial homotopy and another for the derived category of cochains on certain spaces. We also give the stratification of cochains on a space a topological content.
\end{abstract}

\maketitle                  


\section{Introduction}

In~\cite{HopkinsGlobalMethods}, Hopkins constructed a bijection between thick subcategories of compact objects in the derived category $\Derived(R)$ of a commutative Noetherian ring $R$, and specialization closed subsets of the prime ideal spectrum of $R$. This was improved upon in~\cite{NeemanChromaticTower}, where Neeman proved a bijection between the localizing subcategories of $\Derived(R)$ and subsets of the prime ideal spectrum $\spec R$ and used this bijection to deduce Hopkins' result.

Recently \cite{BIKsupport,BIKstratifying} Benson, Iyengar and Krause have provided a general machinery for obtaining such results. They consider a compactly generated triangulated category $\Tri$ which has an action of a graded commutative Noetherian ring $S$. For an object $X \in \Tri$ they define the \emph{support of $X$}, denoted $\supp_S X$, which is a subset of $\spec S$ -- the homogeneous prime ideal spectrum of $S$. If certain conditions are met, $\Tri$ is said to be \emph{stratified by $S$} (see Section~\ref{sec: Stratification of a derived category}). The main consequences of which are that taking support induces an inclusion preserving bijection~\cite[Theorem 4.2]{BIKstratifying}:
\[ \left\{
     \begin{array}{c}
        \text{Localizing}\\
        \text{subcategories of }\Tri\\
     \end{array}
   \right\}
   \longleftrightarrow
   \Big\{\text{subsets of } \supp_S \Tri \Big\}
   \]
and under an additional condition there is also the following bijection~{\cite[Theorem 6.1]{BIKstratifying}:
\[ \left\{
     \begin{array}{c}
        \text{Thick}\\
        \text{subcategories of }\Tri^\compact\\
     \end{array}
   \right\}
   \longleftrightarrow
   \left\{
     \begin{array}{c}
        \text{Specialization closed}\\
        \text{subsets  of } \supp_S \Tri \\
     \end{array}
   \right\}
   \]
where $\Tri^\compact$ is the traingulated subcategory of compact objects in $\Tri$. For example, in~\cite[Theorem 8.1]{BIKstratifying} it is shown that the derived category of a formal commutative dga is stratified by its cohomology, when the cohomology ring is Noetherian.

Let $k$ be a field which is either a prime field $\Int/p$ or the rational numbers $\Q$. By an $\sphere$-algebra we mean a ring-spectrum which is an algebra over the sphere spectrum in the sense of~\cite{EKMM}. Note that any dga $A$ over $k$ can be modelled by some $\sphere$-algebra $R$ whose homotopy groups, denoted $\pi_*R$, are isomorphic to the homology groups of $A$. The following is Theorem~\ref{thm: Stratification of algebra with polynomial homotopy}.
\begin{theorem}
\label{the: First theorem}
Let $R$ be a commutative coconnective $\sphere$-algebra such that $\pi_*R$ is a polynomial ring over $k$ on finitely many generators in even degrees. Then $\Derived(R)$ is stratified by the action of $\pi_*R$.
\end{theorem}

This generalizes \cite[Theorem 5.2]{BIKgroupStratifying}, which shows that the derived category of a formal polynomial dga over a field is stratified by the action of its cohomology ring. Note that the $\sphere$-algebras satisfying Theorem~\ref{the: First theorem} above can be very far from being formal. For example, if $k$ is a prime field the theorem above shows that the derived category of $\chains^*(BSU(n);k)$ -- the dga of singular cochains on the classifying space of $SU(n)$, is stratified by $H^*(BSU(n);k)$ (see~\ref{sub: Classifying spaces}). This dga is neither formal nor is it strictly commutative.

It should be noted however that even when $k=\Int/2$ Theorem~\ref{the: First theorem} still requires the generators to be only in even degrees. This is in contrast with \cite[Theorem 5.2]{BIKgroupStratifying}, where odd degree generators are allowed when $k=\Int/2$.

We now turn to the second stratification result. For a connected space $X$ we denote by $\chains^*X$ the cochains algebra of $X$ with coefficients in $k$. The derived category of $\chains^*X$ has a natural action of the cohomology of $X$ with coefficients in $k$. Benson, Iyengar and Krause have shown that when $k$ is the prime field $\Int/p$ and $G$ is a finite $p$-group then $\Derived(\chains^*(BG))$ is stratified by the action of $H^*(BG;k)$ \cite[Theorem 10.6]{BIKgroupStratifying}. Our second result follows similar lines.

A connected space is called \emph{spherically odd complete intersections (soci)} if it can be built from a space with even degree polynomial cohomology using finitely many spherical fibrations, where these spheres are simply connected and of odd dimension, see Definition~\ref{def: Soci} for details. Theorem~\ref{the: Stratification of cochains on soci} reads:
\begin{theorem}
\label{the: Second theorem}
Let $X$ be an soci space, then $\Derived(\chains^*X)$ is stratified by the canonical action of $H^*(X;k)$.
\end{theorem}

Recall that in the rational case $\chains^*X$ can always be modeled by a commutative dga. So Theorem~\ref{the: Second theorem} can be viewed as extending~\cite[Theorem 8.1]{BIKstratifying} for $k=\Q$. See Example~\ref{sub: A rational example} for an soci space whose cochains algebra cannot be modeled by a formal dga. Note that in the rational case every space with Noetherian rational cohomology and finite dimensional rational homotopy is an soci space; this follows from~\cite[Lemma 8.2 \& Theorem 7.5]{GreenleesHessShamir}.

Before describing the topological aspect of this result we make one last remark regarding stratification. The attraction of stratification, beyond yielding a classification result, is that it automatically implies many other results about the structure of the triangulated category. Such results can be found in the work of Benson, Iyengar and Krause, especially in~\cite{BIKstratifying}. For example, let $R$ be an $\sphere$-algebra and suppose that either $R$ satisfies Theorem~\ref{the: First theorem} or that $R=\chains^*B$ for some soci space $B$. Then~\cite[Theorem 6.3]{BIKstratifying} together with Theorems~\ref{the: First theorem} and~\ref{the: Second theorem} imply that the derived category of $R$ satisfies the telescope conjecture.

Theorem~\ref{the: Second theorem} has a topological content arising from the following trivial observation. Fix a connected space $B$. Recall that a space over $B$ is simply a map $f:X \longrightarrow B$. The map $f$ gives $\chains^*X$ a $\chains^*B$-module structure and in this manner we consider $\chains^*X$ as an object in the derived category $\Derived(\chains^*B)$. We say $X$ is a \emph{finitely fibred space over $B$} if $\chains^*X$ is a compact object in $\Derived(\chains^*B)$. Suppose $f:X \longrightarrow B$ and $g:Y \longrightarrow B$ are two finitely fibred spaces over $B$ and let $\phi:X \to Y$ be a morphism of spaces over $B$, i.e. a map $\phi:X\longrightarrow Y$ satisfying $g\phi=f$. Then it is easy to see that the thick subcategory of $\Derived(\chains^*B)^\compact$ generated by $\chains^*X$ is contained in the thick subcategory generated by $\chains^*Y$.

To employ this observation we define a partial order on finitely fibred spaces over $B$. Roughly speaking this is done by saying $X \preceq Y$ if there is a morphism $\phi:X \to Y$ of spaces over $B$. The full construction of this partially ordered set, denoted $\ffpos/B$, is described in Section~\ref{sec: Finitely fibred spaces over B}. Using the observation above it is easy to construct of a map of partially ordered sets (see Lemma~\ref{lem: Map from finitely fibred poset}):
\[ \ffpos/B \longrightarrow
   \left\{
     \begin{array}{c}
        \text{Thick subcategories}\\
        \text{of }\Derived(\chains^*B)^\compact\\
     \end{array}
   \right\}
\]

When $\Derived(\chains^*B)$ is stratified by the action of the cohomology of $B$, we can give a good description of the fibres of this map. This is done, albeit in an implicit manner, in Theorem~\ref{the: Reduced finitely fibred spaces}. Since this description is quite involved, let us instead describe the main consequence of it.

Let $f:X \longrightarrow B$ be a finitely fibred space over $B$ where $f$ is a fibration. Define $\po X$ to be the homotopy pushout $X \cup_{X \times_B X} X$. This construction yields a finitely fibred space over $B$, see Lemma~\ref{lem: Topological realization of R/Isquared}. Corollary~\ref{cor: Topological consequence of stratification} reads:
\begin{corollary}
Suppose that $\Derived(\chains^*B)$ is stratified by the natural action of $H^*B$ and let $X$ and $Y$ be finitely fibred spaces over $B$. If $\supp_{H^*(B;k)} H^*(X;k) \subset \supp_{H^*(B;k)} H^*(Y;k)$ then for $m>>0$ the projection $\po^m Y \times_B X \to X$ induces an injection on cohomology.
\end{corollary}
Note that the support mentioned in the corollary is the usual support of graded modules over a graded ring.

\subsection*{Organization of this paper}
We start in Section~\ref{sec: Stratification of a derived category} by giving all the necessary background into the work of Benson, Iyengar and Krause.

The following sections collect results necessary for the proofs of the main theorems. In Section~\ref{sec: Algebras arising from localizations} we consider Bousfield localizations of a commutative $\sphere$-algebra. Next, in Section~\ref{sec: Ring objects arising from regular sequences}, we examine ring objects and the localizing subcategories they generate. Such objects arise from regular sequences of non-zero divisors. Section~\ref{sec: Colocalization methods} gives two methods for generating colocalization functors.

The necessary conventions regarding spaces and their cochain algebras are given in Section~\ref{sec: Spaces and cochains}. This is followed by the proofs of Theorems~\ref{thm: Stratification of algebra with polynomial homotopy} and~\ref{the: Stratification of cochains on soci} in Section~\ref{sec: Stratifying spherically odd complete intersections spaces}.

The partially ordered set of finitely fibred spaces over a given base space is studied in~\ref{sec: Finitely fibred spaces over B}. Finally we provide some examples in Section~\ref{sec: Examples}.

\subsection*{Notation and terminology}
Much of the work in this paper is carried out in the derived category of a commutative $\sphere$-algebra $R$, denoted $\Derived(R)$. One should not be deterred by the use of such heavy machinery; the main properties of $\Derived(R)$ we use are that it is a triangulated category with a symmetric monoidal product and that $R$ is a compact generator of $\Derived(R)$. The symmetric monoidal product of $\Derived(R)$ is customarily denoted $-\wedge_R-$, however following~\cite{DwyerGreenleesIyengar} we shall denote this by $-\otimes_R-$. The unit for this tensor product is of course $R$ itself.

The sphere spectrum $\sphere$ is a highly structured topological ring-spectrum constructed in~\cite{EKMM}. Roughly speaking a commutative $\sphere$-algebra is a ring-spectrum which has a multiplication that is associative and commutative up to infinitely many homotopies. Some of the $\sphere$-algebras that come into play in this paper will have a dga counterpart whose derived category is equivalent to that of the original $\sphere$-algebra. However, these dga counterparts will very often not be commutative at all. The main place where the properties of commutative $\sphere$-algebras are employed is in Section~\ref{sec: Algebras arising from localizations} where we use the fact that the Bousfield localization of a commutative $\sphere$-algebra is again a commutative $\sphere$-algebra.

We follow~\cite{BIKsupport} and~\cite{BIKstratifying} in both notation and terminology pertaining to triangulated categories. One exception to this are the \emph{homotopy groups} of an object $X$ of the derived category $\Derived(R)$, defined by
\[ \pi^nX=\hom_{\Derived(R)} (R,\Sigma^n X) = \hom^n_{\Derived(R)} (R,X)\]
where $R$ is a dga or an $\sphere$-algebra. Note that when $R$ is a dga $\pi^*X$ is simply the cohomology of $X$, and when $R$ is an $\sphere$-algebra then $\pi^*X$ is simply a regrading of $\pi_*X$ -- the stable homotopy groups of $X$. We shall follow the convention that superscripts denote codegrees and subscripts denote degrees and therefore $X_{\Box} = X^{-\Box}$.

Throughout this paper we work over a ground field $k$ which is either a prime field or the rational numbers. In particular, all cohomology groups of spaces are taken with coefficients in $k$.

\subsection*{Acknowledgements}
I am grateful to Moty Katzman for his invaluable aid in the proof of Proposition~\ref{pro: DJ spaces which are soci} regarding Davis-Januszkiewicz spaces, and also for John Greenlees who suggested looking into Davis-Januszkiewicz spaces in the first place. I would also like to convey my gratitude to Dave Benson, Srikanth Iyengar and Henning Krause for inviting me to participate in the Oberwolfach seminar "Representations of Finite Groups: Local Cohomology and Support", which inspired this work. 		I am also grateful to Srikanth Iyengar for his very helpful comments.

\section{Stratification of a derived category}
\label{sec: Stratification of a derived category}

Throughout this section $\Derived$ will denote the derived category of a dga or an $\sphere$-algebra $R$. Hence $\Derived$ is a triangulated category with several nice properties: it has arbitrary coproducts and it has a compact generator (see below). In particular, $\Derived$ is compactly generated. This section recalls the definitions and properties of stratification given in~\cite{BIKstratifying}.

\subsection*{Localizing subcategories and localization}
A \emph{thick subcategory} of $\Derived$ is a full triangulated subcategory closed under retracts. A \emph{localizing subcategory} is a thick subcategory that is also closed under taking coproducts. The thick (resp. localizing) subcategory \emph{generated} by a given class of objects $\class$ in $\Derived$ is the smallest thick (resp. localizing) subcategory containing $\class$, this subcategory is denoted by $\thick_\Derived(\class)$ (resp. $\loc_\Derived(\class)$).

\begin{remark}
We shall also employ the following terminology from \cite{DwyerGreenleesIyengar}. For $X$ and $Y$ in $\Derived$ we say that $X$ \emph{builds} $Y$, denoted $X \builds Y$, if $Y \in \loc_\Derived (X)$. We say that $X$ \emph{finitely builds} $Y$ if $Y\in \thick_\Derived(X)$ and denote it by $X \fbuilds Y$.
\end{remark}

An object $C \in \Derived$ is called \emph{compact} if $\hom_\Derived(C,-)$ commutes with coproducts. The full subcategory of compact objects in $\Derived$ will be denoted by $\Derived^\compact$; it is clearly a thick subcategory. Since $\Derived$ is the derived category of $R$, the thick subcategory generated by $R$ is $\Derived^\compact$ and the localizing subcategory generated by $R$ is $\Derived$.

Note that if $\class \subset \Derived^\compact$ is a set of compact objects then
\[ \thick_\Derived(\class) = \loc_\Derived (\class) \cap \Derived^\compact\]
This follows from a result of Neeman~\cite[Lemma 2.2]{NeemanConnectLocalSmash}. In particular this implies that if $X,Y \in \Derived^\compact$ and $X \builds Y$ then $X \fbuilds Y$.

A \emph{localization functor} on $\Derived$ is an exact functor $L: \Derived \to \Derived$ together with a natural morphism $\eta:1_\Derived \to L$ such that $L \eta=\eta L$ and $L\eta$ is a natural isomorphism. One important property of a localization functor $L$ is that $\eta$ induces a natural isomorphism
\[ \hom_\Derived(X,LY) \cong \hom_\Derived(LX,LY) \]
Dually there is the concept of a \emph{colocalization functor}, which is a functor $\Gamma: \Derived \to \Derived$ with a natural morphism $\mu:\Gamma \to 1_\Derived$ such that $\Gamma \mu = \mu \Gamma$ and $\Gamma \mu$ is a natural isomorphism. The natural morphism $\mu$ induces an isomorphism
\[\hom_\Derived(\Gamma X,Y) \cong \hom_\Derived(\Gamma X,\Gamma Y) \]

Given a functor $F: \Derived \to \Derived$ define the \emph{kernel of $F$}, denoted $\Ker F$, to be the full subcategory of $\Derived$ whose objects are those satisfying $FX\cong 0$. The \emph{essential image} of $F$ is the full subcategory $\Im F$ consisting of objects $X$ such that $X \cong FY$ for some $Y$.

It is well known that any localization functor $L$ gives rise to a colocalization functor $\Gamma$ such that for any $X \in \Derived$ there is an exact triangle
\[ \Gamma X \xrightarrow{\mu_X} X \xrightarrow{\eta_X} LX \]
and such that $\Ker L = \Im \Gamma$ and $\Ker \Gamma = \Im L$. Similarly, a colocalization $\Gamma$ gives rise to a localization $L$ having the same properties as above. It is also well known that a localization functor $L$ is determined (up to a unique natural isomorphism) by its kernel $\Ker L$.

\subsection*{Spectrum and support}
The \emph{center} of $\Derived$ is the graded-commutative ring of natural transformations $\alpha:1_\Derived \to \Sigma^n$ satisfying $\alpha \Sigma = (-1)^n \Sigma \alpha$. Following~\cite{BIKstratifying} we say that a Noetherian graded-commutative ring $S$ \emph{acts on $\Derived$} if there is a homomorphism of graded rings from $S$ to the graded-commutative center of $\Derived$. For example, when $R$ is a commutative $\sphere$-algebra there is a canonical action of $\pi_*R$ on $\Derived$ given by sending $f:R \to \Sigma^n R$ to the natural transformation $f\otimes_R -$.

To set the grading, for every element $s \in S^n$ and for any object $X \in \Derived$ we have a morphism $s_X: X \to \Sigma^n X$. The \emph{degree} of $s$ is therefore $|s|=-n$ and the codegree of $s$ is $n$. In this way $\pi^*X$ is naturally a (left) graded $S$-module for every object $X\in \Derived$.

Let $\spec S$ be the partially ordered set of homogeneous prime ideals of $S$. A subset $\vv \subset \spec S$ is \emph{specialization closed} if whenever $\pp \subset \qq$ and $\pp \in \vv$ then also $\qq \in \vv$. Given a specialization closed subset $\vv$, an object $X\in \Derived$ is called \emph{$\vv$-torsion} if $(\pi^*X)_\pp = 0$ for all $\pp \not\in \vv$, where $(\pi^*X)_\pp$ is the usual (homogeneous) localization at $\pp$. Let $\Derived_\vv$ be the full subcategory of $\vv$-torsion objects, this is a localizing subcategory \cite[Lemma 4.3]{BIKsupport}. By \cite[Proposition 4.5]{BIKsupport} there is a localization $L_\vv$ and a colocalization $\Gamma_\vv$ such that
\[ \Gamma_\vv X \to X \to L_\vv X \]
is an exact triangle and $\Ker L_\vv = \Im \Gamma_\vv = \Derived_\vv$. Note that both $\Gamma_\vv$ and $L_\vv$ are \emph{smashing}, i.e. they commute with coprducts~\cite[Corollary 6.5]{BIKsupport}.

\begin{remark}
The original definition of $\vv$-torsion objects given in \cite{BIKstratifying} appears in a different form than the one given here. This is because the definition in~\cite{BIKstratifying} is designed for a triangulated category with a set of compact generators. In our setting, where there is a single compact generator, the definition above agrees with the original definition from~\cite{BIKstratifying}. We further remark on this at the end of this section.
\end{remark}

\begin{remark}
We should note that the traditional topological definition of a smashing localization is different. In topology a localization $L$ is \emph{smashing} if it is equivalent to the functor $A \otimes -$ for some object $A$ and some monoidal product $\otimes$. In the setting we consider the monoidal product is $\otimes_R$ whose unit is $R$. When $S$ acts on $\Derived$ via a ring homomorphism $S \to \pi^*R$, then $L_\vv$ is isomorphic to $L_\vv(R) \otimes_R -$ (see \cite[Section 7]{BIKstratifying}). Since this will be the case in all the situations considered in this paper, we see that $L_\vv$ and $\Gamma_\vv$ are smashing according to both terminologies.
\end{remark}

Given a prime ideal $\pp \in \spec S$ let $\zz(\pp)=\{ \qq \in \spec S \ | \ \qq \nsubseteq \pp \}$. For a homogeneous ideal $\aa$ of $S$ we let $\vv(\aa)=\{\qq \in \spec S \ | \ \aa \subseteq \qq \}$. Both $\zz(\pp)$ and $\vv(\aa)$ are specialization closed subsets of $\spec S$ and hence there are appropriate localizations and colocalizations. For $X$ in $\Derived$ denote by $X_\pp$ the localization $L_{\zz(\pp)}X$. By~\cite[Theorem 4.7]{BIKsupport} $\pi^*(X_\pp) = (\pi^* X)_\pp$, which justifies this notation. Let $\Gamma_\pp$ be the exact functor given by
\[ \Gamma_\pp X = \Gamma_{\vv(\pp)} X_\pp\]
It turns out that $\Im \Gamma_\pp$ is a localizing subcategory of $\Derived$ (see~\cite{BIKstratifying}). For $X \in \Derived$ the \emph{support} of $X$ over $S$ is defined to be
\[ \supp_S X = \{ \pp \in \spec S \ | \ \Gamma_\pp X \neq 0 \} \]

\subsection*{Stratification}
As above, we assume that $S$ acts on $\Derived$. Before defining stratification we need two preliminary concepts from~\cite{BIKstratifying}. First, a \emph{local-global principle holds} for the action of $S$ on $\Derived$ if for every $X\in \Derived$
\[ \loc_\Derived (X) = \loc_\Derived (\{ \Gamma_\pp X \ | \ \pp \in \spec S\} )\]
Note that by~\cite[Corollary 3.5]{BIKstratifying} if $S$ has a finite (homogeneous) Krull dimension then the local-global principle holds. Second, a localizing subcategory of $\Derived$ is \emph{minimal} if it is nonzero and contains no nonzero localizing subcategories.

\begin{definition}[\cite{BIKstratifying}]
\label{def: Stratification}
The triangulated category $\Derived$ is \emph{stratified} by $S$ if the following two conditions hold:
\begin{description}
\item[S1] The local-global principle holds for the action of $S$ on $\Derived$.
\item[S2] For every $\pp \in \spec S$ the localizing subcategory $\Im \Gamma_\pp$ is zero or minimal.
\end{description}
\end{definition}

For a subcategory $\bl$ of $\Derived$ let $\supp_S \bl$ be the union $\bigcup_{X\in \bl} \supp_S X$. One important result of~\cite{BIKstratifying} is the following theorem (translated to our setting):
\begin{theorem}[{\cite[Theorem 4.2]{BIKstratifying}}]
If $\Derived$ is stratified by the action of $S$ then the map of sets
\[ \left\{
     \begin{array}{c}
        \text{Localizing}\\
        \text{subcategories of }\Derived\\
     \end{array}
   \right\}
   \xrightarrow{\supp_S(-)}
   \Big\{\text{subsets of } \supp_S \Derived \Big\}
   \]
is a bijection which respects inclusions. The inverse map sends a subset $\uu \subset \supp_S \Derived$ to the localizing subcategory whose objects are $\{X\in \Derived \ | \ \supp_S X \subset \uu\}$.
\end{theorem}

We will be more interested in the classification of thick subcategories of $\Derived^\compact$. Such classification is given by the following result from~\cite{BIKstratifying} (again translated to our setting):
\begin{theorem}[{\cite[Theorem 6.1]{BIKstratifying}}]
Suppose that $\Derived$ is stratified by the action of $S$ and that $\pi^*R$ is Noetherian. Then the map
\[ \left\{
     \begin{array}{c}
        \text{Thick}\\
        \text{subcategories of }\Derived^\compact\\
     \end{array}
   \right\}
   \xrightarrow{\supp_S(-)}
   \left\{
     \begin{array}{c}
        \text{Specialization closed}\\
        \text{subsets  of }\supp_S \Derived\\
     \end{array}
   \right\}
   \]
is a bijection which respects inclusions. The inverse map sends a specialization closed subset $\vv \subset \supp_S \Derived$ to the full subcategory whose objects are $\{X\in \Derived^\compact \ | \ \supp_S X \subset \vv\}$.
\end{theorem}

\subsection*{Koszul objects}
Given an object $M$ of $\Derived$ and an element $s \in S$ let $M/s$ be the object defined by the following exact triangle
\[ \Sigma^{-|s|} M \xrightarrow{s_M} M \to M/s \]
Obviously, the object $M/s$ is defined only up to isomorphism. For a sequence $s=(s_1,...,s_n) \subset S$ let $M/s$ be the object $M/s_1/s_2\cdots/s_n$.

In the situations we will consider, $\Derived$ has a symmetric monoidal tensor product $\otimes_R$ whose unit is $R$ and $S$ acts on $\Derived$ via a ring homomorphism $S \to \pi^*R$. In such situations it is easy to show that $M/s$ is unique up to isomorphism and that, up to isomorphism, it is independent of the ordering of $s_1,...,s_n$.

\subsection*{Homotopy colimits}
Given a telescope
\[ X_1 \xrightarrow{\varphi_1} X_2 \xrightarrow{\varphi_2} \cdots \xrightarrow{\varphi_{n-1}} X_n \xrightarrow{\varphi_n} \cdots\]
of objects in $\Derived$ its \emph{homotopy colimit} is defined by the exact triangle
\[ \bigoplus_n X_n \xrightarrow{1-\varphi_n} \bigoplus_n X_n \to \hocolim_n X_n\]
Thus for every $i$ we get a morphism $\psi_i:X_i \to \hocolim_n X_n$. Moreover, it is easy to see that $\psi_i = \varphi_i \psi_{i+1}$. The following lemma is very well known.
\begin{lemma}
Let $M$ be a compact object of $\Derived$ and let $X_1 \xrightarrow{\varphi_1} X_2 \xrightarrow{\varphi_2} \cdots$ be a telescope in $\Derived$. Then any morphism $M \to \hocolim_n X_n$ splits through some $X_i$ for $i>>0$.
\end{lemma}

\subsection*{Tensor and cotensor}
Since $R$ is a commutative $\sphere$-algebra its derived category is equipped with a symmetric monoidal tensor product $\otimes_R$ whose unit is $R$. This implies that every localizing subcategory $\al$ of $\Derived$ is a \emph{tensor ideal} localizing subcategory, i.e. it satisfies
\[ X,Y \in \al \ \Rightarrow \ X \otimes_R Y \in \al\]
This allows us to use results pertaining to stratification of tensor ideal localizing subcategories, such as the results in~\cite[Section 7]{BIKstratifying}.

Recall that $\pi^*R$ has a canonical action on $\Derived$, defined using the tensor product. If a ring $S$ acts on $\Derived$ via a ring homomorphism $S \to \pi^*R$, the action of $S$ is called \emph{canonical}. All the actions of a ring on $\Derived$ we will encounter in this paper will be canonical. This in particular implies that the local-global principal (condition \textbf{S1}) for the action of $S$ on $\Derived$ holds~\cite[Theorem 7.2]{BIKstratifying}.

The derived category $\Derived$ also has an internal function object functor $\Hom_R(-,-)$, sometimes called a cotensor. For every $X$ and $Y$ in $\Derived$ the internal function object yields an adjoint pair
\[\hom_\Derived(X\otimes_R Y,Z) \cong \hom_\Derived (X, \Hom_R(Y,Z))\]
and this adjunction is natural in $X$ and $Z$ (see for example~\cite{BIKsupport}). In particular, setting $X=R$ we see that
\[\pi^*\Hom_R(Y,Z) \cong \hom^*_\Derived(Y,Z)\]
and setting also $Y=R$ yields $\Hom_R (R,Z) \cong Z$.

When $R$ is a commutative $\sphere$-algebra then the internal function object is simply the derived version of the function $R$-module spectrum.

\subsection*{On having a single compact generator}
The machinery of~\cite{BIKsupport} and~\cite{BIKstratifying} is designed for stratification of compactly generated categories. Since $\Derived$ has a single compact generator (which is also the unit of the monoidal structure), it allows us to simplify some definitions.

For example, let $\vv \subset \spec S$ be a specialization closed set. In~\cite{BIKstratifying} an object $X\in \Derived$ is defined to be $\vv$-torsion if $\hom^*_\Derived(C,X)_\pp=0$ for all $C\in \Derived^\compact$ and for all $\pp \in \spec S \setminus \vv$ ($\hom^*_\Derived(C,X)$ is an $S$-module in an obvious way). It is easy to see this is equivalent to requiring that $(\pi^*X)_\pp=0$ for all $\pp \in \spec S \setminus \vv$ -- which is the definition given in this paper.

Here is another concept which is simplified in our setting. In~\cite{BIKstratifying} the category $\Derived$ together with the action of $S$ is called \emph{Noetherian} if for every $C\in \Derived^\compact$ the $S$-module $\hom^*_\Derived(C,C)$ is finitely generated. In our setting, this is equivalent to requiring that $\pi^*R$ be finitely generated $S$-module. We leave this claim as an exercise to the reader.

\section{Algebras arising from localizations}
\label{sec: Algebras arising from localizations}

Throughout this section we are assuming that $R$ is a commutative $\sphere$-algebra whose derived category is $\Derived$. In addition we suppose a graded commutative Noetherian ring $S$ acts on $\Derived$ via a ring homomorphism $S \to \pi^*R$. We fix a specialization closed subset $\vv$ in $\spec S$. The object of this section is to prove the following proposition.

\begin{proposition}
\label{pro: Bousfiled localization and spec localization}
The $R$-module $L_\vv R$ is a commutative $\sphere$-algebra and the adjoint functors
\[ \Derived \xrightarrow{L_\vv R\otimes_R -} \Derived(L_\vv R) \quad \text{and} \quad \Derived(L_\vv R) \xrightarrow{G} \Derived \]
where $G$ is the forgetful functor, yield an equivalence between $\Derived(L_\vv R)$ and $\Im L_\vv$. In particular the coaugmented functor $L_\vv$ is naturally isomorphic to the coaugmented functor $G(L_\vv R\otimes_R -)$.
\end{proposition}

This result is hardly surprising. To prove this proposition one needs only to reconcile between Bousfield localizations of $\sphere$-algebras and the language of~\cite{BIKstratifying} which we follow here.

\subsection*{Bousfield localization}
We shall follow the definitions from~\cite[Chapter VIII]{EKMM}, translating them to our context.

Let $E$ be an $R$-module. A morphism $f:X \to Y$ in $\Derived$ is an \emph{$E_*$-equivalence} if $E\otimes_R f$ is an isomorphism in $\Derived$. An object $W\in\Derived$ is \emph{$E_*$-acyclic} if $E\otimes_R W \cong 0$. An object $L$ is called \emph{$E_*$-local} if $\hom_\Derived(f,L)$ is an isomorphism for every $E_*$-equivalence $f$, equivalently if $\hom_\Derived(W,L)=0$ for every $E_*$-acyclic object $W$. Finally, an \emph{$E_*$-localization} of an object $M$ is an $E_*$-equivalence $M \to M_E$ such that $M_E$ is $E_*$-local.

\begin{remark}
By~\cite[Theorem VIII.1.6]{EKMM} for every object $M\in \Derived$ there exists an $E_*$-localization. Since such a localization is unique up to a unique isomorphism, one can make it functorial. Thus, there exists a functor $L_E:\Derived \to \Derived$ and a natural transformation $\varepsilon:1_\Derived \to L_E$ such that $\varepsilon_M:M \to L_E M$ is an $E_*$-localization. It is easy to see that $L_E$ is a localization functor.
\end{remark}

What is more important to us is that the $E_*$-localization of $R$ is again a commutative $\sphere$-algebra, a result shown in~\cite{EKMM}.
\begin{theorem}[{\cite[Theorem VIII.2.2]{EKMM}}]
The localization map $R \to R_E$ can be constructed as a map of commutative $\sphere$-algebras.
\end{theorem}

\begin{remark}
There are two delicate points hidden here. The first is that we need $R$ to be, in the terminology of~\cite{EKMM}, a q-cofibrant commutative $\sphere$-algebra. Since any commutative $\sphere$-algebra is equivalent to a q-cofibrant one, we may assume that our original $R$ is such a commutative q-cofibrant $\sphere$-algebra. The second is that in the original statement of~\cite[Theorem VIII.2.2]{EKMM} it is only for a commutative q-cofibrant $R$-algebra $A$ that the localization $A \to A_E$ is a map of commutative $R$-algebras. But we may apply this theorem to the case where $A$ is a commutative q-cofibrant $R$-algebra replacement of $R$. Hence the resulting unit map $R \to A_E$ is both an $E$-localization and a map of commutative $R$-algebras.
\end{remark}

Thus we have an adjunction
\[ \Derived \xrightarrow{R_E\otimes_R -} \Derived(R_E) \quad \text{and} \quad \Derived(R_E) \xrightarrow{G} \Derived \]
and $\Derived(R_E)$ has a symmetric monoidal structure $\otimes_{R_E}$.

\subsection*{Localization at $\vv$ as a Bousfield localization}
Throughout the rest of this section $E$ is the object $L_\vv R$. By~\cite[Theorem 8.2]{BIKsupport} there is a natural isomorphism
\[ L_\vv(-) \cong L_\vv R \otimes_R -\]
and it is easy to see this is an isomorphism of coaugmented functors. Therefore an object $W \in \Derived$ is in $\Ker L_\vv$ if and only if $W$ is $E_*$-acyclic. Similarly a morphism $f$ is an $E_*$-equivalence if and only if $L_\vv f$ is an isomorphism.

\begin{lemma}
The full subcategory of $E_*$-local objects is equal to $\Im L_\vv$.
\end{lemma}
\begin{proof}
Suppose $M \in \Im L_\vv$ then $\hom_\Derived(X,M) \cong \hom_\Derived(L_\vv X,M)$ for every $X \in \Derived$. Therefore for any $E_*$-acyclic $W$, $\hom_\Derived(W,M)=0$ and $M$ is $E_*$-local. So now suppose $M$ is $E_*$-local. The morphism $M \to L_\vv M$ is an $E_*$-equivalence between $E_*$-local objects and is therefore an isomorphism.
\end{proof}

From this lemma it is obvious that $E_*$-localization is isomorphic to $L_\vv$. In particular the morphism $\varepsilon:R \to R_E$ is isomorphic to $R \to L_\vv R$. It follows that there is a natural isomorphism of coaugmented functors
\[ L_\vv(-) \cong R_E \otimes_R -\]
Note also that $E_*$-localization is given by $\varepsilon\otimes_R 1_M: M \to R_E \otimes_R M$, i.e. $E_*$-localization is smashing in the topological terminology.

Since $\varepsilon:R \to R_E$ can be constructed as a map of commutative $\sphere$-algebras, there are adjoint functors
\[ \Derived \xrightarrow{R_E\otimes_R -} \Derived(R_E) \quad \text{and} \quad \Derived(R_E) \xrightarrow{G} \Derived \]
where $G$ is the forgetful functor. This yields the following well known result (compare~\cite[Proposition VIII.3.2]{EKMM}).

\begin{corollary}
The functors $R_E \otimes_R -$ and $G$ described above induce an equivalence of triangulated categories between $\Im L_\vv$ and $\Derived(R_E)$.
\end{corollary}

This completes the proof of Proposition~\ref{pro: Bousfiled localization and spec localization}. \qed

\section{Ring objects arising from regular sequences}
\label{sec: Ring objects arising from regular sequences}

Ring objects are not rings in any standard sense; these are simply objects in the derived category that have structure morphisms like those of $R$-algebras, though they may lack associativity. Nevertheless, ring objects have sufficient similarity to $R$-algebras that we may use them to understand the structure of the localizing subcategory they generate, as we do in Corollary~\ref{cor: The localizing subactegory generated by a field like is minimal}.

In this section we are assuming that $R$ is a commutative $\sphere$-algebra. As before $\Derived$, the derived category of $R$, has a symmetric monoidal product $-\otimes_R-$ whose unit is $R$. There is the natural unit isomorphism $\unitmap_M: R\otimes_R M \to M$ for every $M\in \Derived$. The role of the natural morphism $\unitmap$ has so far been hidden and will mostly remain so. But one must note that whenever we use an isomorphism $M \cong R\otimes_R M$ we are implicitly referring to $\unitmap$. Since $\otimes_R$ is symmetric there is also a natural twist isomorphism $\tau:X \otimes_R Y \to Y\otimes_R X$.

\subsection*{Ring objects}
Following~\cite[Definition V.2.1]{EKMM} we make the following definitions. A \emph{ring object} in $\Derived$ is an object $A$ in $\Derived$ endowed with a unit morphism $u:R \to A$ and a product morphism $m:A \otimes_R A \to A$ which make $A$ into an algebra with respect to the tensor product $\otimes_R$, i.e. the following diagram commutes
\[ \xymatrixcompile{
{R\otimes_R A} \ar[r]^{u \otimes 1} \ar[dr]_\unitmap & {A\otimes_R A} \ar[d]^m &
{A\otimes_R R} \ar[l]_{1\otimes u} \ar[dl]^{\unitmap\tau} \\
& {A} } \]
A morphism of ring objects is a morphism in $\Derived$ which respects the relevant structures. We emphasize that there is no assumption of commutativity nor associativity of $A$.

A (left) \emph{$A$-module} is an object $M\in \Derived$ endowed with a morphism $a:A\otimes_R M \to M$ satisfying $a(u \otimes 1_M) = 1_M$. Again, we do not assume associativity of this structure.

If $A$ is a ring object in $\Derived$ then there is a $\pi^*R$-bilinear pairing $\pi^*A \times \pi^*A \to \pi^*A$ and a unit map $\pi^*R \to \pi^*A$ making $\pi^*A$ into a (possibly non-associative) $\pi^*R$-algebra. If $A$ is associative then $\pi^*A$ is an associative $\pi^*R$-algebra. Similarly, if $M$ is an $A$-module then $\pi^*M$ is a $\pi^*A$-module which is associative when $M$ is associative.

\begin{lemma}
\label{lem: Ring object maps surject module homotopy}
Let $A$ be a ring object in $\Derived$ and let $M$ be an $A$-module. Then the map
\[ u^*:\hom_\Derived^*(A,M) \to \pi^*M\]
induced by the unit $u:R\to A$ is a surjection.
\end{lemma}
\begin{proof}
Let $a:A\otimes_R M \to M$ be the action morphism of $A$ on $M$. Given a morphism $f:R \to \Sigma^n M$ the following commutative daigram
\[ \xymatrixcompile{
{R} \ar[r]^-\cong \ar[d]^u & {R\otimes_R R} \ar[r]^-{1\otimes f} \ar[d]^{u \otimes 1} & {R \otimes_R \Sigma^n M} \ar[d]^{u \otimes 1}\\
{A} \ar[r]^-\cong & {A\otimes_R R} \ar[r]^-{1\otimes f} & {A \otimes_R \Sigma^n M} \ar[r]^-a & {\Sigma^n M}
}\]
shows that
\[ u^*(a (A\otimes f)) = f\]
\end{proof}

\subsection*{Ring objects from regular sequences}
The main tool we need is the following result from~\cite{EKMM}.

\begin{theorem}[{\cite[Theorem V.2.6]{EKMM}}]
Suppose that $\pi^*R$ is concentrated in even degrees and let $x$ be a regular element in $\pi^*R$. Then $R/x$ has a ring object structure such that the obvious morphism $u_x:R \to R/x$ is the unit morphism.
\end{theorem}
Note that the multiplication morphism $m:R/x \otimes_R R/x \to R/x$ need not be unique. The next result has also been observed in~\cite{EKMM}.

\begin{corollary}
\label{cor: Dividing by a regular sequence may yield a ring object}
Suppose that $\pi^*R$ is concentrated in even degrees and let $X=(x_1,...,x_n)$ be a sequence of non-zero divisors in $\pi^*R$. Then $R/X$ has a ring object structure and the obvious morphism $u:R \to R/X$ is a morphism of ring objects.
\end{corollary}
\begin{proof}
As noted in~\cite[Lemma V.2.2]{EKMM}, if $B$ and $B'$ are ring objects in $\Derived$ then so is $B\otimes_R B'$. Since $R/x_i$ has a ring object structure for every $i$ then so is $R/X=R/x_1 \otimes_R \cdots \otimes_R R/x_n$. Let $u_i:R \to R/x_i$ be the unit morphisms of these ring objects, then clearly
\[ u=u_1 \otimes \cdots \otimes u_n: R \to R/x_1 \otimes_R \cdots \otimes_R R/x_n\]
is the unit morphism of $R/X=R/x_1 \otimes_R \cdots \otimes_R R/x_n$.
\end{proof}

\begin{remark}
Although the multiplication on $R/X$ need not be unique, it is not difficult to show that when $X$ is a regular sequence of non-zero divisors then any choice of multiplication on $R/X$ would yield the standard algebra structure on $\pi^*(R/X)$, i.e. the multiplication on the quotient ring $\pi^*R/(X)$ -- where $(X)$ is the ideal generated by $x_1,...,x_n$.
\end{remark}

\subsection*{Module structure}
In what follows $X=(x_1,..,x_n)$ is a regular sequence of non-zero divisors in $\pi^*R$ and we are assuming that $\pi^*R$ is concentrated in even degrees. For an object $M \in \Derived$ we denote by $M/X$ the object $R/X \otimes_R M$, this agrees with our original definition of $M/X$.

\begin{lemma}
\label{lem: Multiplication by divided element is zero}
For every $x \in (X)$ and for every object $M\in \Derived$ the morphism $M/X \to \Sigma^n M/X$ given by multiplication by $x$ is zero.
\end{lemma}
\begin{proof}
It is enough to prove the lemma for the case where $M=R$. First, let $x\in\pi^*R$ be such that $R/x$ has a ring object structure, then by~\cite[Lemma V.2.4]{EKMM} the morphism $R/x \to \Sigma^n R/x$ given by multiplication by $x$ is the zero morphism. Since $X=(x_1,...,x_n)$ and $R/X=R/x_1 \otimes_R \cdots \otimes_R R/x_n$ we see that multiplication by $x_i$ induces the zero morphism on $R/X$ for all $i=1,...,n$, which completes the proof.
\end{proof}

An immediate consequence of Lemma~\ref{lem: Multiplication by divided element is zero} is the following corollary.

\begin{corollary}
\label{cor: Homotopy of M/X is associative module}
For every object $M\in \Derived$ the $\pi^*R$-module $\pi^*(M/X)$ is naturally a $\pi^*R/(X)$-module.
\end{corollary}

\subsection*{Regular local rings}
We now consider the case where $\pi^*R$ is a graded local ring concentrated in even degrees and $X=(x_1,...,x_n)$ is a regular sequence of non-zero divisors in $\pi^*R$ generating the maximal ideal $(X)$. Hence $\pi^*(R/X)$ is a graded field. Since $\pi^*R$ is generated in even degrees then $R/X$ has a ring object structure whose unit we denote by $u:R \to R/X$.

\begin{lemma}
\label{lem: Kernel of mapping from a field is zero}
Let $f:R/X \to M$ be a morphism in $\Derived$ such that $\pi^*f \neq 0$. Then the kernel of $\pi^*f$ is zero.
\end{lemma}
\begin{proof}
Suppose $y$ is a non-zero element in the kernel of $\pi^*f$. Then there exists $y' \in \pi^*R$ such that $y=u y'$. Consider the morphism $\varphi:\Sigma^n R/X \to R/X$ which is multiplication by $y'$. Clearly $\pi^*(\varphi)(1)=y$. Note that $\varphi$ is invertible -- since $y$ is invertible we can choose an element $z' \in \pi^*R$ such that $y^{-1}=z' u$, clearly multiplication by $z'$ is the inverse of $\varphi$. Hence $\pi^*\varphi$ is an isomorphism and therefore for every $b \in \pi^*(R/X)$ we have that $b=ay$ for some $a \in \pi^* R$. But this implies $\pi^*f = 0$, in contradiction.
\end{proof}

\begin{lemma}
\label{lem: Field module is direct sum}
For every $M \in \Derived$ the object $M/X$ is equal to a direct sum of copies of suspensions of $R/X$. In particular, if $M/X\neq 0$ then $R/X$ is a retract of $M/X$.
\end{lemma}
\begin{proof}
From Corollary~\ref{cor: Homotopy of M/X is associative module} it follows that $\pi^*(M/X)$ is a vector space over the graded field $\pi^*(R/X)$. Choose a basis $B \subset \pi^*(M/X)$. By Lemma~\ref{lem: Ring object maps surject module homotopy} for every $b \in B$ there is a morphism $\Psi_b:\Sigma^{|b|} R/X \to M$ such $(\pi^*\Psi_b)(1)=b$. From Lemma~\ref{lem: Kernel of mapping from a field is zero} we see that $\Ker \pi^*\Psi_b =0$. Hence the morphism:
\[ \bigoplus_{b \in B} \Psi_b: \bigoplus_{b \in B}\Sigma^{|b|} R/X \to M\]
induces an isomorphism on homotopy groups and is therefore an isomorphism.
\end{proof}

\begin{lemma}
\label{lem: hom (A/X,-) is like A/X otimes}
For every object $M$ in $\Derived$ we have that
\[ \hom_\Derived^*(R/X,M)=0 \ \Longleftrightarrow \ M/X =0\]
\end{lemma}
\begin{proof}
Recall that $\Hom_R$ denotes the internal function object in $\Derived$ and that $\Hom_R (R,M) \cong M$. Let $x \in \pi^*R$ and let $\chi:\Sigma^{|x|}R \to R$ be the morphism representing $x$. It is easy to see that $\Hom_R(\Sigma^{-|x|}\chi,M) = \chi \otimes_R M$. Hence the following is an exact triangle
\[ \Sigma^{|x|}\Hom_R(R/x,M) \to \Sigma^{|x|} M \xrightarrow{x} M \to R/x \otimes_R M\]
and so $M/x \cong \Sigma^{|x|+1} \Hom_R(R/x,M)$. Recall that $X=(x_1,...,x_n)$, a simple inductive argument shows that $M/X \cong \Sigma^m\Hom_R(R/X,M)$ for some $m$.
\end{proof}

\begin{remark}
Lemma~\ref{lem: hom (A/X,-) is like A/X otimes} above is in fact true for any commutative $\sphere$-algebra $R$ and for any sequence $x_1,...,x_n \in \pi^*R$.
\end{remark}

\begin{corollary}
\label{cor: The localizing subactegory generated by a field like is minimal}
The localizing subcategory generated by $R/X$ is minimal.
\end{corollary}
\begin{proof}
Let $M \in \loc_\Derived(R/X)$ be a non-zero object, then $\hom^*_\Derived (R/X,M) \neq 0$. From Lemma~\ref{lem: hom (A/X,-) is like A/X otimes} it follows that $M/X \neq 0$ and hence $R/X$ is a retract of $M/X$ by Lemma~\ref{lem: Field module is direct sum}. Since $M \builds M/X$ we see that also $M \builds R$.
\end{proof}

\section{Colocalization methods}
\label{sec: Colocalization methods}

Throughout this section $R$ is a commutative $\sphere$-algebra and as usual $\Derived$ denotes $\Derived(R)$. In this section we give two methods for constructing certain colocalization functors.

\subsection*{The localizing subcategory generated by a Koszul object}
Let $X=(x_1,...,x_n)$ be a sequence of elements in $\pi^*R$. Our ultimate objective here is to describe $\loc_\Derived(R/X)$ in terms of homotopy groups of objects. We will do this by following classical ideas from local cohomology, namely the Grothendieck construction of local cohomology. These ideas can be found in the work of Greenlees and May~\cite{GreenleesMay} (see also~\cite{DwyerGreenlees}) and are by now well known, so there is no claim for originality here.

Given $x \in \pi^*R$ let $K_x$ be the object defined by the exact triangle
\[ K_x \to R \xrightarrow{x} \Sigma^{-|x|} R \]
in other words $K_x =\Sigma^{|x|+1}R/x$. Let $R[1/x]$ be the homtopy colimit of the the telescope
\[ R\xrightarrow{x} \Sigma^{-|x|} R\xrightarrow{x} \Sigma^{-2|x|} R \to \cdots\]
and let $K_x^\infty$ be defined by the exact triangle
\[ K_x^\infty \to R \to R[1/x]\]
For the sequence $X=(x_1,...,x_n)$ of elements in $\pi^*R$ we denote by $K_X^\infty$ the tensor product \[K_{x_1}^\infty \otimes_R \cdots \otimes_R K_{x_n}^\infty\]
and by $K_X$ the object $\Sigma^{-n}R/X=K_{x_1} \otimes_R \cdots \otimes_R K_{x_n}$.

\begin{lemma}
\label{lem: A[1/x] is Kx null}
Let $x \in \pi^*R$, then $R[1/x] \otimes_R K_x = 0$ and $\hom_\Derived^*(K_x,R[1/x]) = 0$.
\end{lemma}
\begin{proof}
Clearly $R[1/x] \otimes_R K_x$ is the homotopy colimit of the telescope
\[ K_x\xrightarrow{x} \Sigma^{-|x|} K_x\xrightarrow{x} \Sigma^{-2|x|} K_x \to \cdots\]
Therefore $\pi^*(R[1/x] \otimes_R K_x)$ is the colimit of
\[ \pi^* K_x\xrightarrow{x} \Sigma^{-|x|} \pi^* K_x\xrightarrow{x} \Sigma^{-2|x|} \pi^* K_x \to \cdots\]
which is zero and so $R[1/x] \otimes_R K_x=0$. The assertion that $\hom_\Derived^*(K_x,R[1/x]) = 0$ now follows from  Lemma~\ref{lem: hom (A/X,-) is like A/X otimes}.
\end{proof}

\begin{lemma}
Let $x \in \pi^*R$, then $K_x$ and $K_x^\infty$ build each other.
\end{lemma}
\begin{proof}
From Lemma~\ref{lem: A[1/x] is Kx null} we see there is an exact triangle
\[ K_x^\infty \otimes_R K_x \to  R \otimes_R K_x \to 0\]
Therefore $K_x$ is isomorphic to $K_x^\infty \otimes_R K_x$ which is built by $K_x^\infty$.

To show that $K_x^\infty$ is built by $K_x$ note that for every $n$ there is an exact triangle
\[K_x \to K_{x^{n+1}} \to K_{x^n}\]
An inductive argument then shows that $K_x$ builds $K_{x^n}$ for all $n$. It is easy to see that $K_x^\infty$ is the homotopy colimit of a telescope
\[ K_x \to K_{x^{2}} \to K_{x^3} \to \cdots\]
and therefore $K_x$ builds $K_x^\infty$.
\end{proof}

The next result is \cite[Proposition 6.1]{DwyerGreenlees} translated to our setting.

\begin{corollary}
Let $X=(x_1,...,x_n)$ be a sequence of elements in $\pi^*R$. Then $R/X$ and $K^\infty_X$ build each other.
\end{corollary}

It follows that for every $M \in \Derived$ the object $K_X^\infty \otimes_R M$ is built by $R/X$. In fact, one can show that $K_X^\infty \otimes_R -$ is the colocalization functor whose image is $\loc_\Derived(R/X)$. However we will not require this result (a similar result can be found in~\cite[Proof of Proposition 9.3]{DwyerGreenleesIyengar}). What we do need is the following characterization of objects built by $R/X$.

\begin{proposition}
\label{pro: Cellular A/X objects are power torsion}
Let $X=(x_1,...,x_n)$ be a sequence of elements in $\pi^*R$ and let $M$ be an object of $\Derived$. Then $M\in \loc_\Derived(R/X)$ if and only if $\pi^*M$ is $(X)$-power torsion.
\end{proposition}
\begin{proof}
Suppose that $\pi^*M$ is $(X)$-power torsion. Then for every $i=1,...,n$, the object $R[1/x_i] \otimes_R M$ is zero and therefore $K_{k_i}^\infty \otimes_R M \cong M$. We conclude that $K_X^\infty \otimes_R M \cong M$ which implies that $R/X$ builds $M$.

On the other hand, the full subcategory of objects whose homotopy is $(X)$-power torsion is a localizing subcategory which contains $R/X$. So if $M \in \loc_\Derived(R/X)$ then $\pi^*M$ is $(X)$-power torsion.
\end{proof}

\subsection*{Colocalization at an algebra}
Let $A$ be an object in $\Derived$ and let $\class$ be $\loc_\Derived(A)$. Denote by $\Gamma_\class$ the colocalization functor whose image is $\class$. We give a construction for the functor $\Gamma_\class$ in the case where $A$ is an $R$-algebra which is compact as an object in $\Derived$. We shall follow the definitions of~\cite{DwyerGreenlees}

\begin{definition}
A morphism $U \to V$ of objects in $\Derived$ is an \emph{$A$-equivalence} if the induced map $\hom_\Derived^*(A,U) \to \hom_\Derived^*(A,V)$ is an isomorphism. An object $N$ is \emph{$A$-null} if $\hom_\Derived^*(A,N)= 0$. An object $M$ is \emph{$A$-colocal} if, for every $A$-null object $N$, $\hom_\Derived^*(M,N) = 0$. If $C \to M$ is an $A$-equivalence and $C$ is $A$-colocal then $C$ is called an \emph{$A$-colocalization of $M$}. It is easy to see that an $A$-colocalization of an object $M$ is unique up to isomorphism.
\end{definition}

It is also easy to see that $\class$ is contained in the class of $A$-colocal objects. For every object $M$ the morphism $\Gamma_\class M \to M$ is clearly an $A$-equivalence. Since an $A$-equivalence between $A$-colocal objects is an isomorphism, we see that every $A$-colocal object $C$ is isomorphic to $\Gamma_\class C$. It easily follows that the class of $A$-colocal objects is $\class$ and that the colocalization $\Gamma_\class M \to M$ is $A$-colocalization.

\begin{definition}
\label{def: R mod I tensor n}
Let $A$ be an $R$-algebra with a unit map $u:R \to A$ and let $I$ be defined by the exact triangle
\[ I \to R \to A\]
Denote by $I^{\otimes n}$ the tensor product
\[ \underbrace{I \otimes_R I \cdots \otimes_R I}_{n \text{ times}} \]
Taking the tensor product of $u$ with itself $n$ times yields a morphism $I^{\otimes n} \to R$. Let $R/I^{\otimes n}$ be defined by the exact triangle
\[ I^{\otimes n} \to R \to R/I^{\otimes n}\]
In particular, $R/I \cong A$.
\end{definition}

The objects $R/I^{\otimes n}$ were defined by Lazarev in~\cite{LazarevAinftyRingSpectra} and shown to be $R$-algebras. Moreover, Lazarev shows there are maps of $R$-algebras
\[ R/I \xleftarrow{a} R/I^{\otimes 2} \leftarrow \cdots \leftarrow R/I^{\otimes (n-1)} \xleftarrow{a} R/I^{\otimes n} \leftarrow \cdots\]
compatible with the unit maps $u_n:R \to R/I^{\otimes n}$. The resulting tower was called a \emph{generalized Adams resolution} in~\cite[Theorem 7.1]{LazarevAinftyRingSpectra}. We shall only need the following map of exact triangles, note that the rightmost column is also exact
\[ \xymatrixcompile{
& & {I^{\otimes n} \otimes_R A} \ar[d] \\
{I^{\otimes(n+1)}} \ar[r] \ar[d] & {R} \ar[r] \ar[d]^= & {R/I^{\otimes(n+1)}} \ar[d]^a \\
{I^{\otimes n}} \ar[r] & {R} \ar[r] & {R/I^{\otimes n}}
}\]

\begin{theorem}
\label{thm: Colocalization as hom R/Itensorn}
Suppose $A$ is an $R$-algebra that is compact as an object in $\Derived$. Then for every object $M$ in $\Derived$ the natural morphism
\[ \hocolim_n \Hom_R(R/I^{\otimes n},M) \to M\]
is an $A$-colocalization.
\end{theorem}
\begin{proof}
Let $P_A(M) = \hocolim_n \Hom_R(I^{\otimes n},M)$, note there is a morphism $M \to P_A(M)$ induced by the maps $I^{\otimes n} \to R$. We first show that if $N$ is $A$-null then $N \to P_A(N)$ is an isomorphism. Since $N$ is $A$-null then
\[ \Hom_R(I^{\otimes n} \otimes_R A,N) \cong \Hom_R(I^{\otimes n},\Hom_R(A,N)) \cong 0\]
Hence the morphisms $\Hom_R(I^{\otimes n},N) \to \Hom_R(I^{\otimes (n+1)},N)$ are isomorphisms and so $N \to P_A(N)$ is also an isomorphism.

Next we show that for any $M$ the object $P_A(M)$ is $A$-null. Since $A$ is compact any morphism $g:\Sigma^i A \to P_A(M)$ splits through $\Hom_R(I^{\otimes n}, M)$ for some $n$. Consider the exact triangle
\[ \Hom_R(A,\Hom_R(I^{\otimes n}, M)) \to \Hom_R(I^{\otimes n}, M) \to \Hom_R(I^{\otimes (n+1)}, M)\]
Clearly any morphism $f:\Sigma^i A \to \Hom_R(I^{\otimes n}, M)$ splits through $\Hom_R(A,\Hom_R(I^{\otimes n}, M))$. This implies that the composition
\[ \Sigma^i A \xrightarrow{f} \Hom_R(I^{\otimes n}, M) \to \Hom_R(I^{\otimes (n+1)}, M)\]
must be zero. In particular $g=0$ and therefore $P_A(M)$ is $A$-null.

Since $A$ is compact then so is $I$ and therefore also $I^{\otimes n}$ is compact for all $n$. From this it follows that the class of modules $M$ for which $P_A(M)\cong 0$ is a localizing class. Denote this localizing class by $\bl$, we shall next show that $A \in \bl$.

Recall that $\pi^*(\Hom_R(I^{\otimes n},A))$ is isomorphic to $\hom^*_\Derived(I^{\otimes n},A)$. Given a morphism $f:\Sigma^i I^{\otimes n} \to A$ it is easy to see that $f$ equals to the composition of the morphism  $1\otimes u:I^{\otimes n} \to I^{\otimes n} \otimes_R A$ with the morphism $m(f\otimes 1): \Sigma^i I^{\otimes n} \otimes_R A \to A$ where $m:A \otimes_R A \to A$ is the multiplication on $A$. The upshot of this is that $\hom^*_\Derived(I^{\otimes n}\otimes_R A,A) \to \hom^*_\Derived(I^{\otimes n},A)$ is a surjection and therefore the map
\[ \pi^*(\Hom_R(I^{\otimes n},A)) \to \pi^*(\Hom_R(I^{\otimes (n+1)},A)) \]
is the zero map. We conclude that $P_A(A)\cong 0$ and hence $\loc_\Derived(A) \subset \bl$.

In what follows we denote $\loc_\Derived(A)$ by $\class$. Let $L_\class$ be the localization functor such that $\Ker L_\class=\class$ and let $\Gamma_\class$ be the corresponding colocalization functor. For any object $M$ we have shown that $P_A(\Gamma_\class M)\cong 0$. By applying $P_A$ to the exact triangle $\Gamma_\class M \to M \to L_\class M$ we see that $L_\class M \cong P_A(M)$. Since there is an exact triangle
\[ \hocolim_n \Hom_R(R/I^{\otimes n},M) \to M \to \hocolim_n \Hom_R(I^{\otimes n},M)\]
the proof is complete.
\end{proof}

Our goal in giving Theorem~\ref{thm: Colocalization as hom R/Itensorn} is the following corollary.

\begin{corollary}
\label{cor: Retract of R/Iotimesn module}
Suppose $A$ is an $R$-algebra which is compact as an object in $\Derived$ and let $M$ be a compact object built by $A$. Then $M$ is a retract of $\Hom_R(R/I^{\otimes n},M)$ for some $n$.
\end{corollary}
\begin{proof}
The obvious morphism $\hocolim_n \Hom_R(R/I^{\otimes n},M) \to M$ is an isomorphism. Since $M$ is compact, its inverse $M \to \hocolim_n \Hom_R(R/I^{\otimes n},M)$ splits through $\Hom_R(R/I^{\otimes \nu},M)$ for some $\nu$. Hence $M$ is a retract of $\Hom_R(R/I^{\otimes \nu},M)$.
\end{proof}

The next two lemmas are needed for Section~\ref{sec: Finitely fibred spaces over B} where we discuss a topological aspect of stratification.

\begin{lemma}
\label{lem: A builds R/I tensor n}
For every $n$ the objects $A$ and $R/I^{\otimes n}$ build each other.
\end{lemma}
\begin{proof}
The maps of $R$-algebras $R/I^{\otimes n} \to A$ show that $R/I^{\otimes n}$ builds $A$ for every $n$. On the other hand, an inductive argument using the exact triangles
\[ I^{\otimes n} \otimes_R A \to R/I^{\otimes (n+1)} \to R/I^{\otimes n} \]
shows $A$ builds $R/I^{\otimes n}$ for every $n$.
\end{proof}

\begin{lemma}
\label{lem: Retract of R/Itensorn when cellular}
Suppose $A$ is an $R$-algebra which is compact as an object in $\Derived$ and let $M$ be a compact $R$-module built by $A$. Then for $n>>0$ the morphism
\[u_n \otimes 1_M: M \to R/I^{\otimes n} \otimes_R M\]
yields an injection on $\pi^*$.
\end{lemma}
\begin{proof}
By Corollary~\ref{cor: Retract of R/Iotimesn module} the object $M$ is a retract of $N=\Hom_R(R/I^{\otimes n},M)$ for $n>>0$. Since $R/I^{\otimes n}$ is an $R$-algebra we see that $M$ is a retract of the $R/I^{\otimes n}$-module $N$. Let $f:M \to N$ and $g:N \to M$ be morphisms in $\Derived$ such that $gf=1_M$.

Let $A_n$ denote the $R$-algebra $R/I^{\otimes n}$. Since $N$ is an $A_n$-module, standard adjunctions show there exists a morphism $h:A_n \otimes_R N \to N$ such that $h(u_n \otimes 1_N) =1_N$. To explain this further, there is an adjunction
\[ \hom_\Derived(N,N) \cong \hom_{\Derived(A_n)}(A_n \otimes_R N,N)\]
and we take $h$ to be the morphism adjoint to $1_N$. Using the forgetful functor we can view $h$ as a morphism in $\Derived$ and since $u_n \otimes 1$ is the unit for this adjunction we see that $h(u_n \otimes 1_N) =1_N$.

Consider the following commutative diagram
\[ \xymatrixcompile{
{M} \ar[rr]^-{u_n \otimes 1_M} \ar@/^/[d]^f & & {A_n \otimes_R M} \ar@/^/[d]^{1_{A_n} \otimes f} \\
{N} \ar@/^/[rr]^-{u_n \otimes 1_N} \ar@/^/[u]^g & & {A_n \otimes_R M} \ar@/^/[u]^{1_{A_n} \otimes g} \ar@/^/[ll]^-h } \]
Diagram chasing reveals that the morphism $(u_n \otimes 1_N) f=(1_{A_n} \otimes g)(u_n \otimes 1_M)$ has a left inverse. Hence $\pi^*(u_n \otimes 1_M)$ is an injection.
\end{proof}

\section{Spaces and cochains}
\label{sec: Spaces and cochains}

Before starting to work with spaces and their cochain algebras we need to establish definitions and notation. We follow conventions from~\cite{DwyerGreenleesIyengar}.

\subsection*{Cochains}
Recall that the field $k$ gives rise to a commutative $\sphere$-algebra denoted $Hk$ whose stable homotopy is $k$ in degree 0 and zero elsewhere. The $\sphere$-algebra $Hk$ is also called an Eilenberg-Mac Lane spectrum. Let $X$ be a connected space. The cochains of $X$ with coefficients in $k$ is the commutative $\sphere$-algebra \[\chains^*(X;k)=F_\sphere(\Sigma^\infty X_+,Hk)\]
where $F_\sphere$ stands for the function $\sphere$-module. We will usually denote the cochains on $X$ by $\chains^*X$ unless we want to emphasize the role of $k$. Similarly, $H^*X$ is $H^*(X;k)$, i.e. the cohomology of $X$ with coefficients in $k$.

Let $A$ be the differential graded algebra of singular cochains on $X$ with coefficients in $k$. In~\cite{ShipleySpectraDGA} Shipley has shown that the derived category of differential graded modules over $A$ is equivalent to $\Derived(\chains^*X)$. However $A$ is usually not commutative. Moreover, in many cases $A$ will not be equivalent to any commutative dga. One exception to this is the case where $k=\Q$ and $X$ is simply connected. Under these conditions $A$ will always be equivalent to a commutative dga.

\subsection*{Spaces}
We will work both with topological spaces and with simplicial sets. Work with simplicial sets will be carried out in Section~\ref{sec: Finitely fibred spaces over B} and we shall comment on the passage between these two equivalent Quillen model categories there. So, until we reach Section~\ref{sec: Finitely fibred spaces over B}, spaces are (compactly generated Hausdorff) topological spaces. Also until section~\ref{sec: Finitely fibred spaces over B} all the spaces we work with are implicitly $p$-complete (when $k$ is a prime field) or rational (when $k=\Q$). Hence when we talk of a sphere we mean the $p$-completion (or rationalization) of a sphere. This in particular implies that a space $X$ which has the homology (with coeffiecients in $k$) of a simply connected sphere is, in our language, weakly equivalent to a sphere.

Following~\cite{DwyerGreenleesIyengar} we say a space $X$ is of \emph{Eilenberg-Moore type (EM-type)} if $X$ is connected, $H^*X$ is of finite type and
\begin{enumerate}
\item $X$ is simply connected when $k=\Q$ or
\item $k$ is a prime field, $\pi_1(X)$ is a finite $p$-group and $X$ is $p$-complete.
\end{enumerate}

Suppose $X \longrightarrow B$ and $Y \longrightarrow B$ are maps of EM-type spaces. We denote by $X \times_B Y$ the homotopy pullback of these maps. An important feature of the fact that these are EM-type spaces is the equivalence of commutative $\sphere$-algebras:
\[ \chains^*(X\times_B Y) \sim \chains^*X \otimes_{\chains^*B} \chains^*Y \]

\section{Stratifying spherically odd complete intersections spaces}
\label{sec: Stratifying spherically odd complete intersections spaces}

In this section we prove two of the main theorems. Before presenting them we require some background and definitions.

A local Noetherian ring R is \emph{complete intersections} if its completion is of the form $Q/(f_1,...,f_c)$ where $Q$ is a regular local ring and $f_1,...,f_c$ is a regular sequence. One way to mimic this is to start with a space which has polynomial cohomology and successively construct spherical fibrations. This is because a spherical fibration yields properties on the cochain algebras of the base and total space which are akin to dividing by a regular element. This approach is a major theme of~\cite{GreenleesHessShamir} and~\cite{BensonGreenleesShamir} and it lies at the heart of the following definition.

\begin{definition}
\label{def: Soci}
A connected space of EM-type $X$ is \emph{spherically odd complete intersections (soci)} if there is a connected space $B$ such that $H^*B$ is a graded polynomial ring on finitely many even degree generators and there are fibrations
\[ S^{n_1} \longrightarrow X_1 \longrightarrow X_0=B, \quad S^{n_2} \longrightarrow X_2 \longrightarrow X_1,\ \ldots \ , \quad S^{n_c} \longrightarrow X_c \longrightarrow X_{c-1} \]
with $X=X_c$ and $n_i$ is odd and greater than 1 for all $i=1,...,c$.
\end{definition}

The two main theorems we shall prove here are the following.

\begin{theorem}
\label{thm: Stratification of algebra with polynomial homotopy}
Let $R$ be a commutative coconnective $\sphere$-algebra such that $\pi_*R$ is a polynomial ring over $k$ on finitely many generators in even degrees. Then $\Derived(R)$ is stratified by the action of $\pi_*R$.
\end{theorem}

\begin{theorem}
\label{the: Stratification of cochains on soci}
Let $X$ be an soci space, then $\Derived(\chains^*X)$ is stratified by the natural action of $H^*X$.
\end{theorem}

\subsection*{Proof of Theorem~\ref{thm: Stratification of algebra with polynomial homotopy}}
We start by showing that the action of $\pi^*R$ on $\Derived$ satisfies condition \textbf{S2} of Definition~\ref{def: Stratification}. Since $\pi^*R$ is a regular ring its localization at a prime is regular local domain. Let $\pp$ be a homogeneous prime ideal of $\pi^*R$. Recall that $R_\pp=L_{\zz(\pp)}$, hence $R_\pp$ can be constructed as a commutative $R$-algbera and $\Im L_{\zz(\pp)}$ is equivalent to $\Derived(R_\pp)$ (Proposition~\ref{pro: Bousfiled localization and spec localization}).

In this paragraph we shall be working in $\Derived(R_\pp)$. Since $\pi^*R_\pp = (\pi^*R)_\pp$ we see that $\pi^*R_\pp$ is a regular local domain and therefore its maximal ideal is generated by a regular sequence of non-zero divisors $X=(x_1,...,x_n)$ in $\pi^*R_\pp$ (see for example~\cite[Exercise 2.2.24]{BrunsHerzog}). Since $\pi^*R_\pp$ is concentrated in even degrees and $X$ is a regular sequence then by Corollary~\ref{cor: Dividing by a regular sequence may yield a ring object} there is a ring object structure on $R_\pp/X$ such that the obvious morphism $R_\pp \to R_\pp/X$ is a morphism of ring objects. Since $\pi^*(R_\pp/X)$ is a graded field, Corollary~\ref{cor: The localizing subactegory generated by a field like is minimal} shows that $\loc_{\Derived(R_\pp)} (R_\pp/X)$ is minimal.

Consider the object $\Gamma_\pp R=\Gamma_{\vv(\pp)} R_\pp$. Since both $L_{\zz(\pp)}$ and $\Gamma_{\vv(\pp)}$ are smashing we see that $\Gamma_\pp R$ generates $\Im \Gamma_\pp$. Moreover, by~\cite[Corollary 6.5]{BIKsupport} $\Gamma_\pp R$ is isomorphic to $R_\pp \otimes_R \Gamma_{\vv(\pp)} R$. In this fashion we may consider $\Gamma_\pp R$ also as an object of $\Derived(R_\pp)$.

The $\pi^*R$-module $\pi^*(\Gamma_\pp R)$ is $\pp$-power torsion because $\Gamma_\pp R \in \Im \Gamma_{\vv(\pp)} = \Derived_{\vv(\pp)}$. It follows that $\pi^*(\Gamma_\pp R)$ is an $(X)$-power torsion $\pi^* R_\pp$-module. By Proposition~\ref{pro: Cellular A/X objects are power torsion} this implies that $\Gamma_\pp R \in \loc_{\Derived(R_\pp)} (R_\pp/X)$ and therefore $\loc_{\Derived(R_\pp)} (\Gamma_\pp R)$ is minimal. From the equivalence of the categories $\Im L_{\zz(\pp)}$ and $\Derived(R_\pp)$ we conclude that $\loc_{\Derived(R_\pp)} (\Gamma_\pp R)$ is equivalent to $\loc_{\Derived} (\Gamma_\pp R)$ and therefore the latter localizing subcategory is also minimal.

By~\cite[Corollary 3.5]{BIKsupport} condition \textbf{S1} is satisfied since the Krull dimension of $\pi^*R$ is finite. \qed

\subsection*{An inductive step}
Here we present an inductive step (Corollary~\ref{cor: Induction step for soci}) which is the main ingredient in the proof of Theorem~\ref{the: Stratification of cochains on soci}. To set the stage we first note some properties arising from a map spaces.

A map $X \longrightarrow Y$ of connected spaces yields a morphism of commutative $\sphere$-algebras $\chains^*Y \to \chains^*X$. This in turn gives rise to the usual forgetful functor
\[ G:\Derived(\chains^*X) \to \Derived(\chains^*Y) \]
which has a left adjoint
\[ F:\Derived(\chains^*Y) \to \Derived(\chains^*X) \]
given by $M \mapsto \chains^*X \otimes_{\chains^*Y} M$. We shall repeatedly use the fact that the preimage under either $F$ or $G$ of a localizing subcategory is a localizing subcategory. Indeed, this is true under any triangulated functor which respects coproducts.

For the rest of this section $f:X \longrightarrow Y$ is a map of EM-type spaces whose homotopy fibre is an odd simply-connected sphere. Let $R$ denote $\chains^*X$ and let $Q$ denote $\chains^*Y$.

\begin{lemma}
If $M\in \Derived(R)$ then $M$ and $R \otimes_{Q} M$ build each other in $\Derived(R)$.
\end{lemma}
\begin{proof}
For any $M \in \Derived(R)$, the object $FG(M)=R\otimes_Q M$ is finitely built by $M$ in $\Derived(R)$. This follows from~\cite[Proposition 10.5]{GreenleesHessShamir} in the rational case and from~\cite[Theorem 7.3]{BensonGreenleesShamir} in the mod-$p$ case.

On the other hand, $R\otimes_Q M$ builds $M$ in $\Derived(R)$. To see this, note that it is enough to show that $R\otimes_Q R$ builds $R$ in $\Derived(R)$. Recall that $\chains^*(X\times_Y X) \simeq R \otimes_Q R$ as $R$-modules. The diagonal map $X \longrightarrow X \times_Y X$ induces a map of $\sphere$-algebras $\chains^*(X\times_Y X) \to \chains^*X$ showing that $\chains^*(X\times_Y X)$ builds $R$ over $\chains^*(X\times_Y X)$. Using one of the projections $X\times_Y X \longrightarrow X$ we see that $\chains^*(X\times_Y X)$ builds $R$ in $\Derived(R)$.
\end{proof}

We now add two assumptions that will be in force until the end of this section. First, that $H^*Y$ is Noetherian and second that $\Derived(Q)$ is stratified by the natural action of $H^*Y=\pi^*Q$.

\begin{lemma}
If $M\in \loc_{\Derived(Q)} R$ then $M$ and $R \otimes_{Q} M$ build each other in $\Derived(Q)$.
\end{lemma}
\begin{proof}
Let $M$ be an object in $\loc_{\Derived(Q)}(R)$. Since $\Derived(Q)$ is stratified by $H^*Y$, then~\cite[Theorem 4.2]{BIKstratifying} shows that $\supp_{H^*Y} M \subset \supp_{H^*Y} R$. Since $H^*Y$ is Noetherian, then by~\cite[Theorem 1.6]{BIKstratifying}
\[ \supp_{H^*Y} (R \otimes_Q M) = \supp_{H^*Y} R \cap \supp_{H^*Y} M = \supp_{H^*Y} M\]
Using~\cite[Theorem 4.2]{BIKstratifying} once more we conclude that $M$ and $R\otimes_Q M$ build each other over $Q$.
\end{proof}

\begin{corollary}
\label{cor: Bijection of localizing subcats for soci}
There is a bijection of partially ordered sets
\[ \left\{
     \begin{array}{c}
        \text{Localizing subcategories}\\
        \text{of }\Derived(R)\\
     \end{array}
   \right\}
   \longleftrightarrow
   \left\{
     \begin{array}{c}
        \text{Localizing subcategories}\\
        \text{of }\loc_{\Derived(Q)}(R)\\
     \end{array}
   \right\} \]
This bijection is given on the one hand by the preimage $G^{-1}$ of the forgetful functor $G:\Derived(R) \to \Derived(Q)$ and on the other hand by $\loc_{\Derived(Q)}(R) \cap F^{-1}$ where $F$ is the left adjoint of $G$.
\end{corollary}
\begin{proof}
Let $\al$ be a localizing subcategory of $\Derived(R)$. Denote by $\al'$ the localizing subcategory $\loc_{\Derived(Q)}(R) \cap F^{-1}(\al)$ and let $\al''$ be $G^{-1}(\al')$. If $M\in \al$ then $R\otimes_Q M \in \al$ and therefore $GM \in \al'$ and $M \in \al''$. If $M \in \al''$ then clearly $R \otimes_Q M \in \al$ and therefore $M \in \al$. We see that $\al=\al''$.

Let $\bl$ be a localizing subcategory of $\loc_{\Derived(Q)}(R)$. Let $\bl'=G^{-1}(\bl)$ and let $\bl''=\loc_{\Derived(Q)}(R) \cap F^{-1}(\bl')$. If $M \in \bl$ then $R \otimes_Q M \in \bl$. Since $R\otimes_Q M \cong GF(M)$ we see that $R\otimes_Q M \in \bl'$ and therefore $M \in \bl''$. If $M \in \bl''$ then clearly $R \otimes_Q M \in \bl$. Since $M$ and $R\otimes_Q M$ build each other over $Q$ we conclude that $M \in \bl$.
\end{proof}

The map $f$ induces a map $g=H^*f:H^*Y \to H^*X$ whose image $g(H^*Y)$ is a sub-algebra of $H^*X$. Hence $g(H^*Y)$ has a canonical action on $\Derived(R)$.

\begin{lemma}
The triangulated category $\Derived(R)$ is stratified by the canonical action of $g(H^*Y)$.
\end{lemma}
\begin{proof}
Since the action of $g(H^*Y)$ is canonical, by~\cite[Theorem 7.2]{BIKstratifying} condition \textbf{S1} of Definition~\ref{def: Stratification} is satisfied. We are left with showing the minimality condition \textbf{S2}.

Let $\pp'$ be a prime in $g(H^*Y)$ and let $\pp \in \spec H^*Y$ be the pre-image of $\pp'$. We claim that under the bijection of Corollary~\ref{cor: Bijection of localizing subcats for soci} the localizing subcategory $\Im \Gamma_{\pp'}$ corresponds to a localizing subcategory of $\loc_{\Derived(Q)}(R) \cap \Im \Gamma_\pp$.

By~\cite[Corollary 5.9]{BIKsupport} an object $M\in \Derived(R)$ is in $\Im \Gamma_{\pp'}$ if and only if $\pi^*M$ is both $\pp'$-local and $\pp'$-torsion. Let $N \in \loc_{\Derived(Q)}(R) \cap F^{-1}(\Gamma_{\pp'})$, then $\pi^*(FN)$ is both $\pp'$-local and $\pp'$-torsion. Hence $\pi^*(R\otimes_Q N)\cong \pi^*(GFN)$ is both $\pp$-local and $\pp$-torsion. As we saw above
\[ \supp_{H^*Y} N = \supp_{H^*Y} (R\otimes_Q N) = \{ \pp \} \]
and therefore $N \in \loc_{\Derived(Q)}(R) \cap \Im\Gamma_\pp$. 

Since $\Im \Gamma_{\pp}$ is minimal or zero, then so are $\loc_{\Derived(Q)}(R) \cap \Im\Gamma_\pp$ and $\Im \Gamma_{\pp'}$.
\end{proof}

\begin{lemma}
\label{lem: X is Noetherian}
The algebra $H^*X$ is finitely generated as a $g(H^*Y)$-module and therefore Noetherian.
\end{lemma}
\begin{proof}
Consider the Serre spectral sequence for the fibration $S^n \longrightarrow X \xrightarrow{\ f \ } Y$. The $E_2$-term of this spectral sequence is the cohomology of $Y$ with local coefficients in the cohomology of $S^n$. If $k=\Q$ then $Y$ is simply connected. Otherwise, since $\pi_1Y$ is a finite $p$-group, the group-algebra $k[\pi_1Y]$ has a unique simple module $k$. We see that in both cases the action of $\pi_1 Y$ on $H^*(S^n)$ is trivial and therefore the $E_2$-term is $H^p(Y;H^q(S^n))$.

We see that the spectral sequence collapse at a finite stage and yields a short exact sequence
\[ g(H^*Y) \to H^*X \to J\]
where $J$ is an $g(H^*Y)$-module which is also a submodule of a free $H^*Y$-module. Since $H^*Y$ is Noetherian we see that $H^*X$ is a finitely generated $H^*Y$-module and therefore Noetherian as a ring.
\end{proof}

We also need the following general result.
\begin{proposition}
Let $A$ be a commutative $\sphere$-algebra. Suppose there are homomorphisms of graded commutative rings $S' \to S \to \pi^*A$ and that $S$ and $S'$ are Noetherian. If the canonical action of $S'$ stratifies $\Derived(A)$ then so does the canonical action of $S$.
\end{proposition}
\begin{proof}
As noted earlier, since the action of $S$ is canonical condition \textbf{S1} is satisfied. We turn to showing the minimality condition \textbf{S2} for the action of $S$.

Let $\pp \subset S$ be a prime ideal and let $\pp' \subset S'$ be the preimage of $\pp$. We will show that $\Im \Gamma_\pp \subset \Im \Gamma_{\pp'}$. For any object $M \in \Derived$ \cite[Corollary 5.8]{BIKsupport} implies that $M\in \Im \Gamma_{\pp}$ if and only if $\pi^*M$ is both $\pp$-local and $\pp$-torsion. It is obvious that if $\pi^*M$ is $\pp$-torsion then it is $\pp'$-torsion. Similarly if $M$ is $\pp$-local then it is also $\pp'$-local. Hence $\Im \Gamma_\pp \subset \Im \Gamma_{\pp'}$ and therefore $\Im \Gamma_\pp$ is minimal or zero.
\end{proof}

We have just proved the inductive step detailed below.
\begin{corollary}
\label{cor: Induction step for soci}
Let $f:X \longrightarrow Y$ be a map of EM-type spaces whose homotopy fibre is an odd simply-connected sphere and suppose that $\Derived(\chains^*Y)$ is stratified by the natural action of $H^*(Y)$ and that $H^*(Y)$ is Noetherian. Then $\Derived(\chains^*X)$ is stratified by the natural action of $H^*X$.
\end{corollary}

\begin{proof}[Proof of Theorem~\ref{the: Stratification of cochains on soci}]
Let $X$ be a soci space, hence there is a simply connected space $B$ such that $H^*B$ is a graded polynomial ring on finitely many even degree generators and there are odd spherical fibrations
\[ S^{n_1} \longrightarrow X_1 \longrightarrow X_0=B, \quad S^{n_2} \longrightarrow X_2 \longrightarrow X_1,\ \ldots \ , \quad S^{n_c} \longrightarrow X_c \longrightarrow X_{c-1} \]
with $X=X_c$. We proceed by induction to show that $H^*X_i$ is Noetherian of finite Krull dimension and that $\Derived(\chains^*X_i)$ is stratified by the natural action of $H^*X_i$.

The induction base is Theorem~\ref{thm: Stratification of algebra with polynomial homotopy} together with the fact that $H^*B$ is graded polynomial on finitely many elements and hence Noetherian and of finite Krull dimension. The induction step is simply Lemma~\ref{lem: X is Noetherian} together with Corollary~\ref{cor: Induction step for soci}.
\end{proof}

\section{Finitely fibred spaces over $B$}
\label{sec: Finitely fibred spaces over B}

In this section spaces are simplicial sets whose category is denoted by $\spaces$. This choice enables us to employ Bousfield's localization of spaces with respect to a homology theory. However, this conflicts with the cochains functor defined earlier, since the cochains functor requires topological spaces as input. Hence we must first translate a simplicial set into a topological space via the geometric realization functor before applying the cochains functor. This will be done implicitly throughout without any further remark.

\subsection*{An overview}
For the rest of this section we fix a field $k$ and an EM-type space $B$. Recall that given a map $f:X \longrightarrow B$ we consider $\chains^*X$ as an object of $\Derived(\chains^*B)$ via the map of $\sphere$-algebras $\chains^*f$. Given a map $g:Y \longrightarrow X$ it is a simple observation that
\[ \loc_{\Derived(\chains^*B)} (\chains^*Y) \subset \loc_{\Derived(\chains^*B)} (\chains^*X)\]
where $\chains^*Y$ is an object of $\Derived(\chains^*B)$ via the composition $fg$.

We will be more interested in EM-type spaces with a given map to $B$ whose cochain algebras are compact objects in $\Derived(\chains^*B)$. Such spaces will be called finitely fibred spaces (Definition~\ref{def: Finitely fibred spaces}). To a finitely fibred space $X$ we will assign the thick subcategory $\Psi(X)=\thick_{\Derived(\chains^*B)}(\chains^*X)$. It follows from the observation above that there is a natural partial order on finitely fibred spaces which makes this assignment into a map of partially ordered sets (Lemma~\ref{lem: Map from finitely fibred poset}). The partially ordered set of finitely fibred spaces will be denoted by $\ffpos/B$ (Definition~\ref{def: Unreduced poset ffsb}).

We will not analyze the image $\Psi$. We will, however, describe the fibres of $\Psi$ when $\Derived(\chains^*B)$ is stratified by the action of $H^*B$; this will be done implicitly in Theorem~\ref{the: Reduced finitely fibred spaces}. To give this description we will define an equivalence relation on $\ffpos/B$ whose equivalence classes yield a new partially ordered set: $\overline{\ffpos/B}$ -- the reduced partially ordered set of finitely fibred spaces. Theorem~\ref{the: Reduced finitely fibred spaces} shows that $\Psi$ splits through $\overline{\ffpos/B}$ and that the induced map from $\overline{\ffpos/B}$ to thick subcategories of $\Derived(\chains^*B)^\compact$ is an injection. Hence the fibres of $\Psi$ are the equivalence classes of $\ffpos/B$ which give $\overline{\ffpos/B}$.

\subsection*{Spaces over a base space}
The category of spaces over $B$, denoted $\spaces/B$ has for objects maps $f:X \longrightarrow B$. A morphism
\[ \phi:(X \xrightarrow{f} B) \to (Y \xrightarrow{g} B)\]
in $\spaces/B$ is a map of spaces $\phi:X \to Y$ such that $g\phi=f$. Though the map $f:X \longrightarrow B$ is an important part of the structure, we will usually denote this object of $\spaces/B$ simply by its underlying space $X$, with the understanding that the map $f$ is implied.

Given a model category structure on the category of spaces $\spaces$ we can lift it to a model category structure on $\spaces/B$, see~\cite{DwyerSpalinky}. The model category on $\spaces$ we shall use is the $k$-local one, in other words we $k$-localize the usual model category structure on spaces as done by Bousfield in \cite{BousfieldLocSpaces}. The main point here is that under the local model category structure the weak equivalences are maps which induce isomorphism on homology with coefficients in $k$. The space $B$ is assumed to be fibrant in the localized model category structure.

Thus we have a model category structure on $\spaces/B$ in which the weak equivalences are $k$-homology isomorphisms. The homotopy category of this model category structure will be denoted by $\homot^k \spaces/B$.

\begin{definition}
\label{def: Finitely fibred spaces}
A \emph{$k$-finitely fibred space over $B$} is an object $f:X \longrightarrow B$ of $\homot^k \spaces/B$ such that $X$ is an EM-type space and $\chains^*X$ is a compact object in $\Derived(\chains^*B)$. Note that we consider $\chains^*X$ as a $\chains^*B$-module via the map of $\sphere$-algebras $\chains^*(f)$.
\end{definition}

\begin{remark}
We should justify the name 'finitely fibred space'. For a map $f:X \longrightarrow B$ denote by $F(f)$ the homotopy fibre of $f$, considered as an $\Omega B$-space. Since
\[ \chains^*F(f) \sim \chains^*X \otimes_{\chains^*B} k\]
we see that if $f:X \longrightarrow B$ is a $k$-finitely fibred space over $B$ then $H^*F(f)$ is a finite dimensional graded $k$-module. The opposite is also true: if $H^*F(f)$ is of finite dimension then $f$ is a $k$-finitely fibred space over $B$ -- this is shown in \cite[Lemma 2.10]{DwyerWilkersonFundamental} for the prime case and the rational case can be proved by similar arguments.
\end{remark}

\subsection*{The partially ordered set of finitely fibred spaces}
We begin by defining the unreduced version of the poset of finitely fibred spaces over $B$.
\begin{definition}
\label{def: Unreduced poset ffsb}
Given $X$ and $Y$ in $\homot^k \spaces/B$ we say that $X \preceq Y$ if there is a morphism $X \to Y$ in $\homot^k \spaces/B$. This relation is clearly reflexive and transitive. If $X\preceq Y$ and $Y\preceq X$ then we say that $X$ is \emph{congruent (over $B$)} to $Y$ and denote this by $X\approx_B Y$. The \emph{poset of finitely fibred spaces over $B$}, denoted $\ffpos/B$, has for elements congruence classes of finitely fibred spaces in $\homot^k \spaces/B$ and the relation $\preceq$ gives the partial order on $\ffpos/B$.
\end{definition}

In what follows we denote by $R$ the $\sphere$-algebra $\chains^*B$ and we denoted by $\Derived$ the derived category of $R$.

\begin{lemma}
\label{lem: Map from finitely fibred poset}
There is a map of partially ordered sets
\[ \Psi: \ffpos/B \longrightarrow
   \left\{
     \begin{array}{c}
        \text{Thick subcategories}\\
        \text{of }\Derived^\compact\\
     \end{array}
   \right\}
\]
given by $\Psi(f:X \longrightarrow B) = \thick_\Derived \chains^*X$.
\end{lemma}
\begin{proof}
If $X$ is isomorphic to $Y$ in $\homot^k \spaces/B$ then there is a finite sequence of $k$-equivalences connecting $X$ and $Y$ in $\spaces/B$. Upon taking cochains these maps become weak equivalences of $\chains^*B$-modules and therefore $\chains^*X$ is isomorphic to $\chains^*Y$ in $\Derived$.

If $X \preceq Y$ then there is a morphism $\phi:X' \to Y'$ in $\spaces/B$ where $X'\cong X$ and $Y' \cong Y$ in $\homot^k \spaces/B$. To see this one need only take $X'$ to be a cofibrant replacement of $X$ and take $Y'$ yo be a fibrant replacement of $Y$. Thus we get a map $\phi:\chains^*(Y') \to \chains^*(X')$ of $\sphere$-algebras over $\chains^*(B)$. Hence $\chains^*(X')$ is a $\chains^*(Y')$-module and therefore $\chains^*(Y')$ builds $\chains^*(X')$ over $\chains^*(B)$. Since both $\chains^*(X')$ and $\chains^*(Y')$ are compact, we see that $\chains^*(Y')$ finitely builds $\chains^*(X')$ over $\chains^*(B)$. Therefore $\thick_\Derived(\chains^*X) \subseteq \thick_\Derived(\chains^*Y)$. This in particular implies that if $X \approx_B Y$ then $\thick_\Derived(\chains^*X) = \thick_\Derived(\chains^*Y)$.
\end{proof}

\subsection*{The reduced partially ordered set of finitely fibred spaces}
Recall that for a finitely fibred space $X$ over $B$, we defined $\po X$ to be the homotopy pushout $X \cup_{X \times_B X} X$. Note that we are assuming, without loss of generality, that $f$ is a fibration of spaces, so $X \times_B X$ has the correct homotopy type. It is easy to see that $\po X$ is naturally a space over $B$. The next lemma shows that $\po X$ is in fact a finitely fibred space over $B$. Recall the algebras $R/I^{\otimes n}$ were defined in Definition~\ref{def: R mod I tensor n}.

\begin{lemma}
\label{lem: Topological realization of R/Isquared}
Let $X$ be finitely fibred space over $B$, let $A$ denote the $R$-algebra $\chains^*X$ and let $I$ be defined by the exact triangle $I \to R \to A$. There is an isomorphism in $\Derived$:
\[ \chains^*(\po X) \cong R/I^{\otimes 2}\]
which commutes with the unit morphisms $R \to R/I^{\otimes 2}$ and $R \to \chains^*(\po X)$.
\end{lemma}
\begin{proof}
In this proof we work in the model category of $R$-modules instead of $\Derived$; this is because we need to use homotopy pushouts. Without loss of generality we also assume that $A$ is a cofibrant $R$-module.

Denote the morphism $I \to R$ by $i$ and let $A'$ denote $R/I^{\otimes 2}$. There are, a priori, two maps $A' \to A$ we consider, these are induced as maps on cones:
\[ \xymatrixcompile{
{I} \ar[r]^i & {R} \ar[r] & {A}\\
{I\otimes_R I} \ar[r]^-{i\otimes i} \ar[u]^{1\otimes i} \ar[d]_{1\otimes i} & {R} \ar[r] \ar[u]^= \ar[d]_=& {A'} \ar@{-->}[d]_{b_1} \ar@{-->}[u]^{b_2} \\
{I} \ar[r]^i & {R} \ar[r] & {A}
}\]

Given maps of $R$-modules $X \to Y \leftarrow Z$ we denote by $X \sqcup_Y Z$ the homotopy pushout of these maps. It is a simple matter to see the diagram above yields an exact triangle of homotopy pushouts
\[ I \sqcup_{I\otimes_R I} I \to R \to A \sqcup_{A'} A\]
It is also a standard that $A \sqcup_{A'} A \simeq A \otimes_R A$. Since homotopy pushout squares are also homotopy bullbacks for $R$-modules, we see there is a homotopy pullback square of $R$-modules
\[ \xymatrixcompile{
{A'} \ar[r]^{b_2} \ar[d]^{b_1} & {A} \ar[d] \\ {A} \ar[r] & {A\otimes_R A}
}\]
On the other hand, applying the cochains functor to the homotopy pushout square
\[ \xymatrixcompile{
{\po X} & {X} \ar[l] \\ {X} \ar[u] & {X\times_B X} \ar[u] \ar[l]
}\]
yields a homotopy pullback square of $R$-modules. Since $\chains^*(X\times_B X) \cong \chains^*X \otimes_{\chains^*B} \chains^*X$ in $\Derived$, we see that $\chains^*(\po X) \cong A'$.

Because both homotopy pullback squares are squares of modules under $R$ (i.e. modules with maps from $R$), the isomorphism $\chains^*(\po X) \cong A'$ commutes with the unit morphisms $R \to A'$ and $R \to \chains^*(\po X)$.
\end{proof}

\begin{definition}
Define a relation $\smile_B$ on $\ffpos/B$. For $X,Y\in \ffpos/B$ we say that $X\smile_B Y$ if:
\begin{enumerate}
\item $Y \approx_B \po X$, or
\item there exists $\phi:X\to Y$ in $\spaces/B$ which induces an injection in cohomology.
\end{enumerate}
We extend $\smile_B$ to be the equivalence relation generated from these conditions. Taking equivalence classes under $\smile_B$ yields a new partially ordered set, $\ffpos'/B$. However, in $\ffpos'/B$ it could still be that for two elements $X$ and $Y$ in $\ffpos'/B$ both $X \preceq Y$ and $Y \preceq X$ but $X \neq Y$. We therefore take equivalence classes again, this time identifying elements $X,Y\in\ffpos'/B$ satisfying $X\preceq Y \preceq X$. The \emph{reduced poset of finitely fibred spaces over $B$}, denoted $\overline{\ffpos/B}$, is the resulting poset.
\end{definition}

\begin{theorem}
\label{the: Reduced finitely fibred spaces}
Suppose that $H^*B$ is Noetherian and that $\Derived$ is stratified by the action of $H^*B$. Then $\Psi$ induces an injection of posets
\[ \Psi: \overline{\ffpos/B} \longrightarrow
   \left\{
     \begin{array}{c}
        \text{Thick subcategories}\\
        \text{of }\Derived^\compact\\
     \end{array}
   \right\}
\]
\end{theorem}

\begin{remark}
We will see below that for any finitely fibred space $X$ its support $\supp_{H^*B} \chains^*X$ is equal to the support of $H^*X$ as a module over $H^*B$. In light of this, Theorem~\ref{the: Reduced finitely fibred spaces} is a statement relating the category of finitely fibred spaces over $B$ with the support of their cohomology as modules over $H^*B$. Now it is hardly surprising that for two finitely fibred spaces $X$ and $Y$ over $B$, if $\supp_{H^*B} H^*X$ is not contained in $\supp_{H^*B} H^*Y$ then there can be no map $X \to Y$ of spaces over $B$. What is new in Theorem~\ref{the: Reduced finitely fibred spaces} is the description of the relation between $X$ and $Y$ when $\supp_{H^*B} H^*X \subset \supp_{H^*B} H^*Y$, which is essentially Corollary~\ref{cor: Topological consequence of stratification} below.
\end{remark}

To prove Theorem~\ref{the: Reduced finitely fibred spaces} we will need two lemmas. In both we are assuming that $H^*B$ is Noetherian and $\Derived$ is stratified by the action of $H^*B$.

\begin{lemma}
\label{lem: Support of finitely fibred spaces}
Let $X \in \ffpos/B$ then $\supp_{H^*B} \chains^*X$ is equal to the support of $H^*X$ as an $H^*B$-module.
\end{lemma}
\begin{proof}
Since $H^*B$ is Noetherian and $\chains^*X$ is a compact object in $\Derived$ we see that $H^*X$ must be a finitely generated $H^*B$-module. By~\cite[Theorem 5.4]{BIKsupport} this implies that $\supp_{H^*B} \chains^*X$ is equal to the small support of $H^*X$ as an $H^*B$-module (see~\cite{BIKsupport} for the definition of the small, or cohomological, support). Since $H^*X$ is a finitely generated module, its small support is the same as the usual support, i.e the set $\{\pp \in \spec H^*B \ | \ (H^*X)_\pp \neq 0 \}$.
\end{proof}

\begin{lemma}
\label{lem: Psi yields cohomology injective map}
Let $X$ and $Y$ be finitely fibred spaces over $B$ such that $\Psi(X) \subset \Psi(Y)$. Then for $m>>0$ the projection $\po^m Y \times_B X \to X$ induces an injection on cohomology.
\end{lemma}
\begin{proof}
Let $A=\chains^*Y$ and define $I$ by the exact triangle $I \to R \to A$. Since $\chains^*X$ is compact and built by $A$, then by Lemma~\ref{lem: Retract of R/Itensorn when cellular} for some $n>>0$ the morphism $\chains^*X \to R/I^{\otimes n} \otimes_R \chains^*(X)$, given by tensoring with the unit morphism $R\to R/I^{\otimes n}$, induces an injection on $\pi^*$. We may choose any $n$ which is large enough, so we choose $n$ such that $n=2^m$ for a large enough $m$.

Applying Lemma~\ref{lem: Topological realization of R/Isquared} repeatedly we arrive at a space $\po^m Y$ over $B$ such that the map $\chains^*B \to \chains^*(\po^m Y)$ is isomorphic to the unit morphism $v:R\to R/I^{\otimes 2^m}$. The projection $p:\po^m Y \times_B X \longrightarrow X$ yields a morphism $\chains^*(p)$ which is is isomorphic in $\Derived$ to $v \otimes 1_{\chains^*(X)}$ and so, by Lemma~\ref{lem: Retract of R/Itensorn when cellular}, $p$ induces an injection on homology.
\end{proof}

The upshot of the last two lemmas is the following corollary mentioned in the introduction.

\begin{corollary}
\label{cor: Topological consequence of stratification}
Suppose that $H^*B$ is Noetherian and that $\Derived(\chains^*B)$ is stratified by the natural action of $H^*B$. Let $X$ and $Y$ be finitely fibred spaces over $B$. If $\supp_{H^*B} H^*X \subset \supp_{H^*B} H^*Y$ then for $m>>0$ the projection $\po^m Y \times_B X \to X$ induces an injection on cohomology.
\end{corollary}

\begin{proof}[Proof of Theorem~\ref{the: Reduced finitely fibred spaces}]
We first show that for $X$ and $Y$ in $\ffpos/B$ if $X \smile_B Y$ then $\Psi(X) = \Psi(Y)$. This is evidently true if $X \approx_B Y$. Suppose $\phi:X \longrightarrow Y$ is a map in $\spaces/B$ which induces an injection on cohomology. We already saw that $\Psi(X) \subseteq \Psi(Y)$ and therefore $\supp_{H^*B} \chains^*(X) \subseteq \supp_{H^*B} \chains^*(Y)$. Since $\phi$ induces an injection on cohomology it follows from Lemma~\ref{lem: Support of finitely fibred spaces} that $\supp_{H^*B} \chains^*(Y) \subseteq \supp_{H^*B} \chains^*(X)$. Since $\Derived$ is stratified by $H^*B$ this implies that $\Psi(Y) \subseteq \Psi(X)$.

From Lemma~\ref{lem: Topological realization of R/Isquared} and Lemma~\ref{lem: A builds R/I tensor n} we see that $\Psi(X)=\Psi(\po^m X)$ for every $m$. All together we so far have shown that $\Psi$ descends to a well defined map of posets from $\ffpos'/B$ to thick subcategories of $\Derived^\compact$. To show that $\Psi$ descends further to define a map from $\overline{\ffpos/B}$ to thick subcategories of $\Derived^\compact$ is an easy exercise.

It remains to show that the resulting map $\Psi$ is an injection. Suppose that $\Psi(X) \subset \Psi(Y)$ then it suffices to show that $X \preceq Y$ in $\overline{\ffpos/B}$, but this follows immediately from Lemma~\ref{lem: Psi yields cohomology injective map}.
\end{proof}

\section{Examples}
\label{sec: Examples}

\subsection{A trivial example}
It is worthwhile noting that for an EM-type space $B$ with finite dimensional cohomology the derived category $\Derived(\chains^*B)$ is minimal as a localizing subcategory. We sketch an argument for this fact. First, $k$ is a generator for $\Derived(\chains^*B)$, this is because $k$ builds $\chains^*B$ by~\cite[Example 5.5]{DwyerGreenleesIyengar}. Hence if $M$ is a non-zero object in $\Derived(\chains^*B)$ then $k \otimes_R M$ is not zero. Since $k \otimes_R M$ is an object of $\Derived(k)$ it is a coproduct of copies of suspensions of $k$. In particular this implies that
\[ M \builds k\otimes_R M \builds k\]

\subsection{Classifying spaces}
\label{sub: Classifying spaces}
The cohomology of the classifying space $BSU(n)$ is the graded polynomial ring $k[x_1,...,x_n]$ where the codegree of $x_i$ is $2i$. So by Theorem~\ref{thm: Stratification of algebra with polynomial homotopy} the derived category $\Derived(\chains^*BSU(n))$ is stratified by the action of the cohomology of $BSU(n)$.

Note that in the rational case this already follows from~\cite[Theorem 5.2]{BIKgroupStratifying}, since the rational cochain algebra on $BSU(n)$ is formal.

\subsection{A non-formal rational example}
\label{sub: A rational example}
This example is taken from~\cite[Example A.6]{GreenleesHessShamir}. Let $X$ be the rational space whose Sullivan algebra has the form
\[ R=(\Lambda(u_2,v_2,x_3,y_3);\ dx=u^2 ,\ dy=uv)\]
where subscripts denote the codegrees of these generators. There is a rational fibration sequence
\[ S^3 \times S^3 \longrightarrow X \longrightarrow \mathbb{C}P^\infty \times \mathbb{C}P^\infty \]
showing that $X$ is an soci space. Hence Theorem~\ref{the: Stratification of cochains on soci} holds and $\Derived(\chains^*X)$ is stratified by the cohomology of $X$. This cohomology ring is
\[ H^*X = \Q[u,v,t]/(u^2,uv,ut,t^2)\]
where the codegree of $t$ is 5.

Note that $R$ is not formal and therefore does not fall under~\cite[Theorem 8.1]{BIKstratifying}. We sketch an argument showing that $R$ is not equivalent to any formal dga. It is shown  in~\cite{GreenleesHessShamir} that $\ext^n_R(\Q,R)$ is $\Q$ when $n=4$ and is zero elsewhere. Suppose, by way of contradiction, that $R$ is equivalent to the formal dga $A=H^*R$. Then there is a map $f:\Sigma^{-2}\Q \to A$ sending the generator of $\Q$ to $u$. Clearly $f$ is a map of dg-$A$-modules and so it is a non-trivial element of $\ext^2_A(\Q,A)$, in contradiction.

\subsection{Davis-Januszkiewicz spaces with complete intersections cohomology}
We shall give here a quick review of these spaces, for a comprehensive review of this material see~\cite{BuchstaberPanovBook}.

Let $K$ be an abstract simplicial complex on the set $[m]=\{1,...,m\}$. Let $BT^m$ be the classifying space of the $m$-torus, i.e. $(\mathbb{C} P^\infty)^m$. For $i \in [m]$ denote by $BT_i$ the $i$'th factor of $BT^m$. Given a subset $\sigma \subseteq [m]$ denote by $BT_\sigma$ the product $\prod_{i\in \sigma} BT_i$. Finally, define the \emph{Davis-Januszkiewicz space} associated to $K$ as:
\[ DJ(K)=\bigcup_{\sigma \in K} BT_\sigma \ \subset \ BT^m\]

Let $k[x]$ denote the polynomial ring $k[x_1,...,x_m]$ where the generators are in codegree 2. Let $I_K$ be the ideal of $k[x]$ generated by the monomials corresponding to simplexes not in $K$. The ring $k(K)=k[x]/I_k$ is the \emph{graded Stanley-Reisner ring} of $K$. It is well known that the cohomology of $DJ(K)$ is isomorphic to the graded Stanley-Reisner ring of $K$.

There is the following naturality property. A map $f:K \to K'$ of simplicial complexes on $[m]$ induces a map $f:DJ(K) \longrightarrow DJ(K')$ which commutes with the inclusion in $BT^m$. The induced map on the cohomology rings $k(K') \to k(K)$ is the appropriate map of Stanley-Reisner rings.

\begin{proposition}
\label{pro: DJ spaces which are soci}
Let $K$ be a simplicial complex on $[m]$ and let $I_K$ be the appropriate ideal of $k(K)$. Suppose that $I_K$ is generated by a regular sequence $(y_1,..,y_n)$. Then $DJ(K)$ is soci and the derived category $\Derived(\chains^*DJ(K))$ is stratified by the graded Stanley-Reisner ring $k(K)$ - which is a complete intersections ring.
\end{proposition}
\begin{proof}
We start with a bit of notation. Given a sequence $\alpha=(\alpha^1,...,\alpha^m)$ of natural numbers denote by $x^\alpha$ the monomial $x_1^{\alpha^1} x_2^{\alpha^2}\cdots x_m^{\alpha^m}$. On the other hand, given a simplex $\sigma \in [m]$ we define $x^\sigma$ to be the monomial $x^\alpha$ where $\alpha^j=1$ if $j\in \sigma$ and otherwise $\alpha^j=0$.

The proof goes by induction on $n$ -- the number of generators in a regular sequence generating $I_K$. For $n=0$ the simplicial complex $K$ is empty and therefore $H^*DJ(K)$ is a graded polynomial algebra.

Suppose $n>0$ and $I_K=(x^{\alpha_1},...,x^{\alpha_j})$, without loss of generality we can assume that $x^{\alpha_1},...,x^{\alpha_l}$ are a minimal generating set for $I$ (by removing extraneous monomials). In addition we know that $I_K=(y_1,...,y_n)$ and that $y_1,...,y_n$ is a regular sequence. This implies that $y_1,...,y_n$ is a minimal generating set for $I$. Hence $l=n$.

Let $\kdim$ denote the Krull dimension of a ring. Because of the minimality of $n$:
\[\kdim k(K)=\kdim k[x]/(y_1,...,y_n)=m-n\]
Now for any commutative $k$-algebra $A$, $\kdim A/(x^\alpha)$ is either $\kdim A -1$, if $x^\alpha$ is regular, or $\kdim A$ otherwise. This implies that $x^{\alpha_1},...,x^{\alpha_n}$ is a regular sequence.

We next show that for $i\neq j$ the monomials $x^{\alpha_i}$ and $x^{\alpha_j}$ have disjoint sets of variables appearing in them (which severely limits the shape of the simplicial complex $K$). The reason is quite simple, suppose $i<j$ and suppose $x_t$ appears in both monomials, then $y=x^{\alpha_i}/x_t$ satisfies $y \cdot x^{\alpha_j} \in (x^{\alpha_i})$ which contradicts the fact that the sequence is regular.

Given $\alpha \in \{0,1\}^m$ let the simplex $\sigma(\alpha) \in [m]$ be defined by $\sigma(\alpha)=\{j| \alpha^j=1\}$. Note that for every $i$, $\alpha_i$ is an element of $\{0,1\}^m$, because $x^{\alpha_i}$ was induced by a simplex of $[m]$. Now define a new simplicial complex $K'$ on $[m]$ by:
\[K'=K\cup \{\sigma \ | \ x^\sigma\in (x^{\alpha_n}) \} \]

The following properties are obvious. First $K \subset K'$, which implies there is a natural map (in fact an inclusion) $DJ(K) \longrightarrow DJ(K')$. Second $I_{K'}=(x^{\alpha_1},...,x^{\alpha_{n-1}})$ and the sequence $(x^{\alpha_1},...,x^{\alpha_{n-1}})$ is a regular sequence.

We conclude that $H^*(DJ(K))=H^*(DJ(K'))/(x^{\alpha_n})$ and $x^{\alpha_n}$ is a regular element of $H^*(DJ(K'))$. Let $F$ be the homotopy fiber of the map $DJ(K) \longrightarrow DJ(K')$. Using the Eilenberg-Moore spectral sequence it is easy to see that the cohomology of $F$ is isomorphic to the cohomology of a sphere of degree $2|\alpha_n|+1$ where $|\alpha|=\Sigma_i \alpha^i$. Thus, after appropriate localization of all the spaces involved, we have an odd spherical fibration:
\[ S^{2|\alpha_n|+1} \longrightarrow DJ(K) \longrightarrow DJ(K') \]
By the induction assumption $K'$ is soci, hence so is $K$.
\end{proof}

\bibliographystyle{amsplain}    
\bibliography{bib2010}          

\end{document}

%% file: DefinitionsSpaceSupport.tex
\usepackage{amsmath,amsfonts,amsthm,amssymb}
\usepackage{fullpage}
\usepackage[svgnames]{xcolor}
\usepackage{hyperref}
\usepackage[all]{xy}

\theoremstyle{plain}    

\newtheorem{theorem}[subsection]{Theorem}
\newtheorem{lemma}[subsection]{Lemma}
\newtheorem{proposition}[subsection]{Proposition}
\newtheorem{corollary}[subsection]{Corollary}

\theoremstyle{definition}   
\newtheorem{definition}[subsection]{Definition}

\theoremstyle{remark}   
\newtheorem{remark}[subsection]{Remark}

\newcommand{\Int}{\mathbb{Z}}   
\newcommand{\sphere}{\mathbb{S}}
\newcommand{\Q}{\mathbb{Q}}     

\newcommand{\Derived}{\mathsf{D}}   
\newcommand{\Tri}{\mathsf{T}}       
\newcommand{\loc}{\mathrm{Loc}}     
\newcommand{\thick}{\mathrm{Thick}} 
\newcommand{\compact}{\mathsf{C}}   
\newcommand{\bl}{\mathsf{B}}        
\newcommand{\al}{\mathsf{A}}        
\newcommand{\class}{\mathsf{A}}     
\newcommand{\Ker}{\mathrm{Ker}}     
\renewcommand{\Im}{\mathrm{Im}}     
\newcommand{\spaces}{\mathcal{S}}   
\newcommand{\homot}{\mathsf{Ho}}    
\newcommand{\ffpos}{\mathsf{F}}     

\newcommand{\Hom}{\mathcal{H}om}
\newcommand{\ext}{\text{Ext}}   
\newcommand{\chains}{\mathrm{C}}    
\newcommand{\hocolim}{\mathrm{hocolim}} 
\newcommand{\spec}{\mathrm{Spec}}   
\newcommand{\supp}{\mathrm{supp}}   
\newcommand{\unitmap}{\ell}         
\newcommand{\po}{\mathcal{PO}}               

\newcommand{\uu}{\mathcal{U}}       
\newcommand{\vv}{\mathcal{V}}       
\newcommand{\zz}{\mathcal{Z}}       
\newcommand{\pp}{\mathfrak{p}}      
\newcommand{\qq}{\mathfrak{q}}      
\renewcommand{\aa}{\mathfrak{a}}    
\newcommand{\builds}{\vdash}        
\newcommand{\fbuilds}{\models}      
\newcommand{\kdim}{\text{Kdim}}     

